\documentclass[reqno]{amsart}
\usepackage{graphicx}

\usepackage{natbib}
\usepackage{bbm,amsmath, amsfonts, amsthm, amssymb, mathrsfs}
\usepackage{bm}
\usepackage[margin=3.1cm]{geometry}
\usepackage{xcolor}
\usepackage{graphicx,tabularx,enumitem}
\usepackage{mathtools}
\graphicspath{{Figures/}}

\usepackage{setspace}
\usepackage{fancyvrb} 

\usepackage{amsaddr}

\DeclareRobustCommand{\SkipTocEntry}[5]{}

\usepackage{etoolbox}
\patchcmd{\section}{\scshape}{\bfseries\large}{}{}
\patchcmd{\subsubsection}{\itshape}{}{}{}
\makeatletter
\def\@seccntformat#1{\csname the#1\endcsname.\space}
\makeatother

\usepackage[textsize=tiny]{todonotes}
\newcommand{\DD}{\mathbb{D}}
\newcommand{\EE}{\mathbb{E}}

\newcommand{\NN}{\mathbb{N}}
\newcommand{\PP}{\mathbb{P}}
\newcommand{\RR}{\mathbb{R}}

\newcommand{\D}{\mathrm{d}}
\newcommand{\ds}{\mathrm{d}s}
\newcommand{\dr}{\mathrm{d}r}
\newcommand{\dt}{\mathrm{d}t}
\newcommand{\du}{\mathrm{d}u}

\newcommand{\Df}{\bm{\mathrm{D}}}

\newcommand{\Bb}{\mathcal{B}}
\newcommand{\Cc}{\mathcal{C}}
\newcommand{\Dd}{\mathcal{D}}

\newcommand{\Ff}{\mathcal{F}}
\newcommand{\Hh}{\mathcal{H}}

\newcommand{\Kk}{\mathcal{K}}

\newcommand{\Tt}{\mathcal{T}}

\newcommand{\Xx}{\mathcal{X}}

\newcommand{\ep}{\varepsilon}

\newcommand{\bx}{\bm{x}}
\newcommand{\bv}{\bm{v}}
\newcommand{\bX}{\bm{X}}
\newcommand{\bY}{\bm{Y}}
\newcommand{\bZ}{\bm{Z}}

\newcommand{\scrD}{\mathscr{D}}
\newcommand{\scrI}{\mathscr{I}}
\newcommand{\scrM}{\mathscr{M}}

\newcommand{\Wh}{\widehat{W}}

\newcommand{\half}{\frac{1}{2}}
\newcommand{\one}{\mathbbm{1}}
\newcommand{\abs}[1]{\left\lvert#1\right\rvert}
\newcommand{\norm}[1]{\left\lVert#1\right\rVert}

\newcommand{\sft}{{\mathsf{t}_0}}
\newcommand{\KH}{K_{H}}
\newcommand{\Ks}{K_\sigma}

\newcommand{\notthis}[1]{}
\newtheorem{theorem}{Theorem}[section]
\newtheorem{corollary}[theorem]{Corollary}
\newtheorem{lemma}[theorem]{Lemma}
\newtheorem{proposition}[theorem]{Proposition}

\theoremstyle{definition}

\newtheorem{remark}[theorem]{Remark}
\newtheorem{assumption}{Assumption}

\newtheorem{assumptionalt}{Assumption}[assumption]

\newenvironment{assumptionp}[1]{
  \renewcommand\theassumptionalt{#1}
  \assumptionalt
}{\endassumptionalt}

  {}

\theoremstyle{plain}

\numberwithin{equation}{section}

\definecolor{ocean}{rgb}{0,0.1,0.6}

\definecolor{imperialGreen}{RGB}{2,137,59}
\definecolor{imperialBlue}{RGB}{0, 62, 116}
\definecolor{imperialBrick}{RGB}{165,25,0}
\definecolor{imperialProcess}{RGB}{0,133,202}

\def\imperialBrick#1{\textcolor{imperialBrick}{#1}}

\usepackage{hyperref}
\hypersetup{colorlinks=false, linkbordercolor=purple,
citebordercolor=imperialProcess}
\usepackage{cleveref}

\title{Functional integration by parts formulae for stochastic Volterra processes}
\author{Alexandre Pannier}
\address{Université Paris Cité, Laboratoire de Probabilités Statistique et Modélisation}
\email{pannier@lpsm.paris}
\date{\today}
\thanks{\emph{Acknowledgments:} I would like to thank Tom Cass for introducing me to the Bismut--Elworthy--Li formula and Ioannis Gasteratos for many insightful discussions.}

\begin{document}

\begin{abstract}
We investigate integration by parts (IBP) formulae for stochastic Volterra equations and we establish the smoothing effect of the expectation. Due to the inherent path-dependent dynamics of this class of processes, standard Bismut--Elworthy--Li (BEL) formulae and lifting procedures fail to produce representations for directional derivatives with respect to the initial curve. We exhibit a new type of fractional IBP for these derivatives which, by means of the Riemann--Liouville fractional derivative, interpolates between the standard chain rule and a pure BEL formula with Cameron--Martin path directions. 
Our assumptions describe precisely the trade-off between the direction's and the test function's regularities. Crucially, we reveal that more roughness leads to more smoothing: 
for a power-law kernel with Hurst parameter~$H\in(0,1/2)$, we show that the expectation is differentiable along constant directions provided that the test function has H\"older continuity~$\beta>2H$.
The proof of the formula relies on a careful analysis of the conditional expectation's temporal regularity and on the well-posedness of its Riemann--Liouville derivative. 
We complement these results with a BEL formula along all square integrable directions whenever the noise is additive, a second order BEL formula and an application to forward and rough volatility models. In the latter case, the derivative is interpreted as the sensitivity with respect to the whole initial forward variance curve.
\end{abstract}

\maketitle

\tableofcontents

\section{Introduction}

Various questions pertaining to the law of stochastic processes can be tackled through the medium of Integration By Parts (IBP) formulae: existence and regularity of a density, gradient bounds and smoothing of the Markov semigroup, strong Feller property, sensitivity analysis (the so-called ``Greeks'' in finance). In the sphere of non-Markovian processes, these questions remain largely unexplored. Therein lie Stochastic Volterra Equations (SVE), which display temporal dependencies that make them suitable for modeling phenomena exhibiting memory. They extend Stochastic Differential Equations (SDE) by introducing kernels in the dynamics which capture the history of the process:
\begin{align}\label{eq:SVE_intro}
    \bX^{\bx}_t = \bx(t) + \int_\sft^t K_b(t,s)b(s,\bX^{\bx}_s)\ds + \int_{\sft}^t K_\sigma(t,s)\sigma(s,\bX^{\bx}_s)\D W_s,\quad  t\in\Tt:=[\sft,T],
\end{align}
where~$0\le\sft\le T$, $\bx:\Tt\to\RR^d$ lies in a Banach space $\Xx$, the kernels $K_b,K_\sigma:\RR_+^2\to\RR^{d\times d}$ are respectively one and two times locally integrable, $b$ and $\sigma$ are measurable maps and $W$ is an $m$-dimensional $(\Ff_t)_{t\ge0}$-Brownian motion.

Stochastic Volterra processes include fractional Brownian motions, anomalous diffusions \cite{henry2010introduction}, generalised Langevin equations \cite{li2017fractional,kupferman2004fractional} and rough volatility models \cite{bayer2023rough}, hence their range of application spans from physical systems to financial modeling, also touching upon turbulent flow velocities \cite{barndorff2008stochastic, chevillard2017regularized}, spatially disordered systems \cite{chen2015fractional} and electricity price modeling \cite{barndorff2013modelling}. They also arise as fundamental equations out of the Mori--Zanzwig formalism, a dimension reduction method where the so-called ``unresolved'' variables are incorporated into a memory kernel~\cite{kupferman2004fractional,givon2005existence}. These applications in turn fueled theoretical advances pertaining to well-posedness and invariance properties \cite{abi2019affine,abi2024polynomial,hamaguchi2025weak}, limit theorems \cite{jacquier2022large,gehringer2022functional,friesen2024volterra,hamaguchi2026exponential} and connections with PDEs \cite{viens2019martingale,wang2022path, bonesini2023rough, gasteratos2025kolmogorov}, to name just a few.

\subsection{Goals.}\label{subsec:goals} The first objective of this work is to identify measurable test functions~$\phi:\Xx\to\RR$, directions $h\in\Xx$ and weights $\pi\in L^2(\Omega;\RR)$ such that the IBP formula
\begin{align}\label{eq:IBP_intro}
    D\Phi(\bx)(h) = \EE[\phi(\bX^{\bx}) \pi]
\end{align}
holds, where $\Phi(\bx) := \EE[\phi(\bX^{\bx})]$ and $D\Phi(\bx)(h)$ is a directional derivative of~$\Phi$ along the direction~$h$ (it is a Fréchet derivative in the remainder of this work). The following two aspects of IBP formulae, widely discussed in the literature, motivate our study:
\begin{description} 
    \item[Smoothing] One would like~$\Phi$ to be differentiable even when~$\phi$ is not, which would be a testimony of the smoothing effect of integration against the law of the Volterra process (or marginal law if~$\phi(\bx)=\varphi(\bx(T))$ for some~$\varphi:\RR^d\to\RR$, which we call the state-dependent case henceforth).
    \item[Tractability]  Of utmost importance for further applications is to identify a ``computable" weight~$\pi$. In the case of standard Bismut--Elworthy--Li (BEL) formula~\cite{bismut1981martingales,elworthy1994formulae}, this weight takes the form of a stochastic integral.
\end{description}
As a specifity of this work, we will also ask whether the test function can be path-dependent. Furthermore, the additional degree of freedom awarded by the choice of~$h$ interrogates the range of admissible directions. This latter consideration drives our journey and leads us to not one formula but a family thereof. They are fractional variants of~\eqref{eq:IBP_intro} and take the form
\begin{align}\label{eq:fIBP_intro}
    D\Phi(\bx)(h) = \EE\Big[\scrD^\alpha_{T-}\big(\phi(\bX^{\bx})-\EE[\phi(\bX^{\bx}) \lvert\Ff_{\cdot}] \big)(\sft) \pi_\alpha \Big], \quad \alpha\in(0,1/2),
\end{align}
where the linear operator $\scrD^\alpha_{T-}$ denotes the right Riemann--Liouville fractional derivative, which definition is given in Section~\eqref{eq:def_fracD_right}. The parameter~$\alpha$ installs a trade-off between the regularity of~$\phi$ and the regularity of~$h$, thereby characterising the smoothing effect mentioned above. Formula~\eqref{eq:fIBP_intro} also bridges the gap between~\eqref{eq:IBP_intro} as $\alpha\to0$ and the standard chain rule when~$\phi\in C^1(\Xx)$.

\subsection{Contributions}\label{subsec:contributions}

Standard Malliavin calculus techniques, such as~\cite{fournie1999applications}, are based on a link between the Malliavin derivative~$\Df_s \bX^{\bx}_t$ and the tangent process~$\bY^{h}_t$ (the Fréchet derivative of~$\bx\mapsto\bX^{\bx}_t$ along~$h$) at all times~$s\le t$. Instead, the starting point of our approach is the simple observation that it is sufficient to find~$\eta\in L^2(\Omega\times\Tt;\RR^m)$ such that~$\langle \Df \bX^{\bx}_t, \eta\rangle_{L^2(\Tt;\RR^m)} = \bY^{h}_t$. Indeed, an application of the integration by parts of Malliavin calculus then yields
\begin{align}\label{eq:proof_comp_intro}
    D\Phi(\bx)(h) = \EE\big[D\phi(\bX^{\bx})(\bY^{h})\big]
    &= \EE\big[\langle \Df \phi(\bX^{\bx}), \eta\rangle_{L^2(\Tt;\RR^m)}\big]
    =\EE\left[\phi(\bX^{\bx}) \int_{\sft}^T \big\langle \eta_s,\D W_s\big\rangle\right],
\end{align}
where~$\langle\cdot,\cdot\rangle$ denotes the inner product in~$\RR^m$. 
This identity between these two functional derivatives is fulfilled at all times~$t\in\Tt$ by setting~$h(t) =\int_{\sft}^t K_\sigma(t,s)h^*(s)\ds$ and~$\eta_s=\sigma(s,\bX^{\bx}_s)^{-1} h^*(s)$. For~\eqref{eq:proof_comp_intro} to make sense, we make the ellipticity assumption~$\sup_{t,x}\abs{\sigma(t,x)^{-1}}\le C$ and we let~$h^*$ be square integrable. The space of such directions~$h$, henceforth called~$\Hh$, is nothing else than the Cameron--Martin space associated with the Gaussian process~$(\int_{\sft}^t K_\sigma(t,s)\D W_s)_{t\in\Tt}$. Three main types of formulae unfold from this approach, each with a different balance between the assumptions. \medskip

\textbf{(IBP)} The first one, which can be found in Theorem~\ref{th:main_IBP}, is the functional IBP of the form
\begin{align}\label{eq:IBP_intro2}
    D\Phi(\bx)(h) = \EE\left[\phi(\bX^{\bx}) \int_{\sft}^T \big\langle (\sigma(s,\bX^{\bx}_s))^{-1} h^*(s),\D W_s\big\rangle\right].
\end{align}
The space of admissible directions is $\Hh$, hence the expectation does have a regularisation effect in those directions and we demonstrate it by extending the class of test functions for which~\eqref{eq:IBP_intro2} holds. We rely for this procedure on the Stone--Weierestrass theorem of~\cite[Proposition 2.6]{cuchiero2026ramifications} which holds on non-locally compact spaces. This result requires the introduction of a weighted space, described in Section~\ref{subsec:weighted_space}, which includes in particular continuous test functions~$\phi:\Cc^\gamma(\Tt)\to\RR$ with polynomial growth for any~$\gamma$ small enough, where~$\Cc^\gamma(\Tt)$ is the space of $\gamma$-H\"older continuous paths.

To the best of our knowledge, Equation~\eqref{eq:IBP_intro2} is the first functional IBP for Volterra processes. It also sheds new light on the semimartingale case with path-dependent $\phi$ as it complements the work of~\cite{jazaerli2017functional}. In contrast with the latter, where the so-called vertical derivative of the functional~$\Phi$ perturbs the path only at the point~$\sft$, our formula computes the derivative with respect to the entire initial curve~$\bx$. 

 \smallskip

\textbf{(f-IBP)} Since $h^*$ is compelled to be square integrable to give meaning to~\eqref{eq:IBP_intro2}, $\Hh$ does not include constant paths other than zero. Our main contribution, Theorem~\ref{thm:frac_IBP}, aims at enhancing the range of admissible directions by trading some regularity from the test function and back to the weight. Parting ways with tradition, we interfere with the Malliavin IBP~\eqref{eq:proof_comp_intro} by introducing a power-law in the weight:
\begin{align*}
    \EE\big[\langle \Df \phi(\bX^{\bx}), \eta\rangle_{L^2(\Tt;\RR^m)}\big]
    = \EE\left[\int_{\sft}^T (s-\sft)^{-\alpha}\big\langle\EE\left[\Df_s \phi(\bX^{\bx}) \lvert\Ff_s\right],  \D W_s \big\rangle\int_{\sft}^T  (s-\sft)^{\alpha}\big\langle \eta_s,\D W_s\big\rangle\right],
\end{align*}
and we identify the first stochastic integral as the Riemann--Liouville derivative
\begin{align*}
    \int_{\sft}^T (s-\sft)^{-\alpha} \D M_s = \Gamma(1-\alpha)\scrD^\alpha_{T^-}(M_T-M_\cdot)(\sft),
\end{align*}
where~$M_t=\EE[\phi(\bX^{\bx})\lvert\Ff_t]$. 
The main technical aspect of the proof, Lemma~\ref{lemma:DalphaM_is_Malpha}, provides sufficient conditions under which it is a legitimate square integrable random variable\footnote{This object is related to the ``fractional martingale'' introduced in~\cite{hu2009fractional}.}. More precisely, when the test function is state-dependent $\phi(\bx)=\varphi(\bx(T))$ and~$\varphi\in\Cc^\beta(\RR^d)$, $\beta\in(0,1)$, we prove through a fine analysis of $M$'s regularity that a sufficient condition is~$\alpha<\beta/2$. The subsequent family of fractional IBP formulae~\eqref{eq:fIBP_intro}, for $0\le\alpha<\beta/2$, can be viewed as an interpolation between the chain rule~$\EE[D\varphi(\bX^{\bx}_T)\bY^h_T]$ when~$\varphi\in C^1(\RR^d)$ (as $\beta\to1$) and the first IBP formula~$\EE\left[\varphi(\bX^{\bx}_T)\int_{\sft}^T \langle\eta_s,\D W_s\rangle\right]$ when~$\varphi$ merely has polynomial growth (as $\beta\to0$).  
As expected, this creates a trade-off between the smoothness of the test function and the range of admissible directions, which is enlarged for all~$\alpha<\beta/2$ to
\begin{align}
    \Hh_{\alpha} = \left\{\int_{\sft}^\cdot K_\sigma(\cdot,s)h^*(s)\ds \;\Big\lvert \int_{\sft}^T s^{2\alpha} \abs{h^*(s)}^2\ds <\infty\right\} \cap C(\Tt;\RR^d).
\end{align}
As $\alpha$ decreases, $\Hh_\alpha$ shrinks (until $\Hh_0=\Hh$) and the class of admissible test functions grows. In particular, in the case of a power-law kernel~$K_\sigma(t,s)=(t-s)^{H-1/2}$, constant functions belong to~$\Hh_\alpha$ for all $\alpha>H$ and thus~$\Phi$ is differentiable along constant directions for all~$\beta>2H$. Note that the Hurst exponent~$H\in(0,1)$ corresponds (almost) to the H\"older regularity of~$\bX^{\bx}$'s trajectories while the space~$\Hh_\alpha$ grows larger as~$H$ becomes small. Therefore this is a situation where rougher paths (equivalently more explosive kernels) leads to more regularisation, analogously to the problem of regularisation by fractional noise~\cite{catellier2016averaging,anzeletti2023regularisation}. 

By twisting a little the definition of the functional derivative, we exhibit in Corollary~\ref{coro:fIBP_singular} that one can even consider unbounded directions such as~$h(t)=(t-\sft)^{\gamma-1/2}$ for $\gamma>1/2+H-\alpha$. For such a class of directions and any $\alpha<\beta/2$, we can also illustrate the derivative's regularity through the following bound (see \eqref{eq:gradient_bound})
\begin{align}\label{eq:gradient_bound_intro}
    \frac{\abs{D\Phi(\bx)(h)}}{h(T)} 
    \lesssim (T-\sft)^{H(\beta-1)}.
\end{align}
This ties in with the notion of fractional smoothness described in \cite{geiss2011fractional}. 

Unfortunately, differentiating in the direction of the kernel~$K_\sigma(\cdot,\sft)$ remains out of reach simply because the solution~$h^*$ to~$K_\sigma(t,\sft) = \int_{\sft}^t K_\sigma(t,s)h^*(s)\ds$ for all~$t\in\Tt$ is a Dirac mass (and hence not an admissible stochastic integrand). \smallskip

\textbf{(add-IBP)} We show in Proposition~\ref{eq:additive_IBP} how to circumvent this last issue. When the test function is state-dependent, it is only necessary to have the identity~$\langle \Df\bX^{\bx}_T,\eta\rangle_{L^2(\Tt;\RR^m)}=\bY^{h}_T$ and, under the further condition of additive noise ($\sigma$ independent of~$\bX^{\bx}$), it is possible to select a specific~$\eta$ that achieves this goal. This method, first used in~\cite{fan2015stochastic}, yields a functional IBP for any square integrable direction including the kernel~$K_\sigma(\cdot,\sft)$. The motivation for differentiating in this direction is presented in Section~\ref{sec:motivation}.
\medskip

The three IBP formulae we presented meet the objectives we set in Section~\ref{subsec:goals} with different sets of assumptions, compromising between the range of test functions, admissible directions and type of noise. In summary:
\begin{description}
    \item[Smoothing] The smoothing effect is characterised for directions ranging from~$\Hh$ to~$L^2(\Tt)$ and passing through~$\Hh_\alpha$ for all~$\alpha\in(0,1/2)$. In the latter case, differentiability is guaranteed for test functions with H\"older regularity~$\beta>2\alpha$. 
    For power-law kernel and direction, we describe the smoothing as an interplay between the parameters and we complement it with the gradient bound~\eqref{eq:gradient_bound_intro}.
    \item[Tractability] The weight~$\pi_\alpha=\int_{\sft}^T (s-\sft)^\alpha \big\langle \sigma(s,\bX^{\bx}_s)h^*(s),\D W_s\big\rangle$ arises by directly exploiting the deterministic direction~$h$. This approach bypasses the invertibility of the Malliavin derivative and of the Malliavin matrix. As a consequence it does not feature auxiliary processes as in the classical BEL formula and can be computed in a straigthforward way.
\end{description}
For convenience, we include in Table~\ref{tab:Summary} a comparison of the three formulae along with their necessary conditions.
\begin{table}[ht!]
\bgroup
\def\arraystretch{1.2}
 \begin{tabularx}{1.02\textwidth}
 {|p{0.15\textwidth} | p{0.24\textwidth} | p{0.19\textwidth} |p{0.18\textwidth} | p{0.14\textwidth}|}
 \hline
 \textbf{IBP formula} & \textbf{Adm. test functions} & \textbf{Adm. directions} & $h(t)=(t-\sft)^{\gamma-\frac12}$ & \textbf{Noise}\\
 \hline
 Theorem \ref{th:main_IBP} \par  \textbf{(IBP)}
 & Path-dep. \& continuous 
 \par 
 \emph{or} 
 state-dep. \& $L^p(\Omega)$
 & $\Hh$
 & $\gamma>1/2+H$
 & multiplicative
 \\
 \hline
 Theorem \ref{thm:frac_IBP}\par  \textbf{(f-IBP)}
 & State-dep. \& $\beta>0$-H\"older
 & $\Hh_\alpha,\,\alpha<\beta/2$
 & $\gamma>1/2+H-\alpha$
 & multiplicative\\
 \hline 
 Proposition \ref{prop:IBP_additive}\par  \textbf{(add-IBP)}
 & State-dep. \& $L^p(\Omega)$
 & $L^2(\Tt;\RR^d)$
 & $\gamma>0$
 & additive\\
 \hline
 \end{tabularx}
 \egroup
\caption{Comparison of the main IBP formulae along with their assumptions. The mention~$L^p(\Omega)$ refers to the condition that~$\varphi(\bX^{\bx}_T)\in L^p(\Omega)$ for some~$p>2$.}
\label{tab:Summary}
\end{table}

\smallskip
\textbf{(2-IBP)} Along the same lines as Theorem~\ref{th:main_IBP}, an IBP formula for the second order derivative is presented in Proposition~\ref{prop:IBP_2}. Its proof requires to study the Fréchet and Malliavin derivatives of the weight~$\pi$, for once again we need to spot the link between them. When the dust settles, the formula holds for any two directions~$g,h\in\Hh$. Second order IBP formulae exist for the geometric Brownian motion \cite{fournie1999applications}, continuous semimartingales \cite[Theorem 2.3]{elworthy1994formulae}, semimartingales driven by jumps \cite[Theorem 3]{cass2006bismut} and in functional form under the Dupire derivative \cite[Theorem 4.17]{jazaerli2017functional}. This is however the first such instance for Volterra processes. \smallskip

\textbf{(SV-IBP)} Eventually, we return to one of the early motivations of this work: rough volatility models. The asset price dynamics we consider are
\begin{equation}\label{eq:SV_model_intro}
\left\{
    \begin{aligned}
    &\D S_t^{\bv} = S_t^{\bv} \zeta(V_t^{\bv}) \D W_t, \quad S^{\bv}_0>0, \\
    &V_t^{\bv}=\bv_t + \int_{\sft}^t K_b(t,s)b(s,V_s^{\bv})\ds + \int_{\sft}^t K_\sigma(t,s)\sigma(s,V_s^{\bv})\D B_s,   
    \end{aligned}
\right.
\end{equation} 
where $W$ and $B$ are Brownian motions such that~$\EE[W_t B_t]= \rho t$ for a correlation parameter~$\rho\in(-1,1)$. This model also includes a wide range of forward variance models such as the Bergomi and Heston classes (whether of Volterra type or not). The initial datum~$\bv$ plays the role of the initial forward variance curve, a path that can be constructed from the (indirectly) quoted variance swaps. Being an infinite-(or high-)dimensional object, finite-difference methods are ill-suited to study the sensitivity with respect to this parameter. Therefore the following IBP formula, stated in more details in Proposition~\ref{prop:IBP_roughvol}, provides a convenient packaging for the derivative of the option price~$\Phi(\bv)=\EE[\phi(S^{\bv})]$:
\begin{align}\label{eq:SVIBP_intro}
    D\Phi(\bv)(h)= \EE\left[\phi(S^{\bv}) \int_{\sft}^T \frac{h^*(s)}{\sqrt{1-\rho^2} \sigma(s,V_s^{\bv})} \D W_s\right].
\end{align}
Formula~\eqref{eq:SVIBP_intro} yields an exact, unbiased Monte Carlo estimator for pathwise sensitivities without hinging on finite-difference approximations. 
This functional derivative is effectively a Greek that, as far as we are aware, has not been given a name yet. It is nevertheless related to the notion of ``forward Vega'' used to denote the sensitivity to a shock in the volatility instruments (e.g. variance swaps) used to hedge one's position.

\subsection{The motivation for functional derivatives}\label{sec:motivation}
The feedback from its past states makes a stochastic Volterra process inconsistent in time when viewed only from its own finite-dimensional state space. For an illustration, let us look at the following example where we set~$\bx(t)\equiv x_0$ and~$b\equiv0$ and try to restart the process at a later time~$\sft>0$ with the same dynamics:
\begin{equation}\label{eq:flow_intro}
    \begin{aligned}
    \bX^{\bx}_t 
    &= x_0 + \int_0^t K_\sigma(t,s) \sigma(\bX^{\bx}_s)\D W_s \\
    &= \underbrace{x_0 + \int_0^{\sft}  K_\sigma(t,s) \sigma(\bX^{\bx}_s)\D W_s}_{=:\widetilde{\bX}^{\sft,\bx}_t} + \int_{\sft}^t K_\sigma(t,s) \sigma(\bX^{\bx}_s)\D W_s.
\end{aligned}
\end{equation}
This suggests that consistency can only be preserved by plugging in an \emph{initial curve} $(\widetilde{\bX}^{\sft,\bx}_t)_{t\ge0}$ which is different from $\bX_{\sft}^{\bx}$, even if~$\bx$ is constant. Under standard assumptions that guarantee pathwise uniqueness of the SVE we can indeed recover the flow property~$\bX^{0,\bx} = \bX^{\sft,\widetilde{\bX}^{\sft,\bx}}$, where~$\bX^{\sft,\bx}$ solves the SVE started at time~$\sft$ with initial curve~$\bx$. Consequently, for any measurable~$\varphi:\RR^d\to\RR$ there exists a measurable deterministic function~$u:[0,T]\times\Xx\to\RR$ such that~$\EE[\varphi(\bX_T^{\bx})\lvert\Ff_{\sft}] = u(\sft,\widetilde{\bX}^{\sft,\bx})$. The latter can be interpreted as a Markov property in the infinite-dimensional space~$\Xx$. The auxiliary process~$\widetilde{\bX}^{\sft,\bx}$ is thus the infinite-dimensional state variable that adequately captures the history of~$\bX^{\bx}$ and, for consistency at~$\sft=0$, the initial path~$\bx$ replaces the starting point~$x_0$ of a Markovian SDE. 

The observation that an augmented version of the Volterra process (called the \emph{lift} in the literature) retrieves the Markov property in infinite dimensions dates back to~\cite{carmona1998fractional} and~\cite{hairer2005ergodicity} and the ramifications of this idea have seen a tremendous inflation in the last few years; we refer the interest reader to~\cite{gasteratos2025kolmogorov} and the references therein. The more pathwise considerations are spelled out by Viens and Zhang \cite{viens2019martingale} who exploited the flow property to derive a functional Itô formula for $f(\sft,\widetilde{\bX}^{\sft,\bx})$. This is further developed in~\cite{gasteratos2025kolmogorov} to complete the rigorous connection between SVEs and PDEs by identifying $u$ as the unique solution to a Kolmogorov equation on a Hilbert space. Both these works, although taking different paths, highlight the central importance of the first and second order directional derivatives of $u$ with respect to $\bx$ along the direction of the kernel~$K_\sigma$ in both the Itô formula and the PDE. Such derivatives along unbounded directions are defined as limits of derivatives along continuous approximations~$K^\delta$. 

The existence of such derivatives is usually obtained via the chain rule by exploiting the regularity of the test function~$\varphi\in C^2(\RR^d)$ and, for the second derivative, under the additional condition that~$K_\sigma(\cdot,t)\in L^4([t,T])$\footnote{In the power-law case~$K_\sigma(t,s)=(t-s)^{H-1/2}\one_{s<t}$ the latter condition restricts the range of parameters to $H\in(1/4,1)$, see Remark 5.9 of \cite{gasteratos2025kolmogorov} for an intuitive explanation. This condition can be dropped if the noise is additive.}. The existence of alternative representations for these derivatives, such as~\eqref{eq:IBP_intro} or~\eqref{eq:fIBP_intro}, may allow to remove this constraint and to consider a larger class of test functions. In attempting to break free from smooth test functions, the question of the differentiability of~$u$ becomes intimately linked to the smoothing effect alluded to above. In the context of Markov processes one speaks of the smoothing of the Markov semigroup and of the ensuing strong Feller property, which importance in stochastic analysis is difficult to overstate.  

From an applied point of view, such differentials also tie in with the sensitivity analysis of option prices, a crucial aspect for risk analysis purposes. Under a forward or rough volatility model  such as~\eqref{eq:SV_model_intro} these sensitivities correspond to the functional derivative of the option price with respect to the initial forward variance curve. The computation of these ``Greeks'' gave rise to a considerable research effort as they allow for faster and more accurate computation via Monte Carlo methods, especially when the payoff is not differentiable and finite-difference methods fail~\cite{fournie1999applications,fournie2001applications,benhamou2001application,gobet2003computation, gobet2004revisiting, gobet2005sensitivity, cass2006bismut}.

\subsection{Relevant literature}

The aforementioned time-inconsistency problem becomes troublesome when trying to apply the classical approaches for IBP formulae: the martingale method of Elworthy--Li \cite{elworthy1994formulae} and the Malliavin calculus technique presented in \cite{fournie1999applications}. When $K_b=K_\sigma\equiv1$, both follow a similar recipe based on a martingale representation~$\varphi(\bX^{\bx}_T)=\EE[\varphi(\bX^{x}_T)] + \int_0^T \theta_s\D W_s$ and lead to a weight of the form~$\pi=\int_0^T \eta_s \D W_s$, for two adapted square integrable process~$\theta$ and~$\eta$. The former representation relies on the Itô formula, which is available for the SVE lift (see \cite[Theorem 4.18]{gasteratos2025kolmogorov}) and thus retrieves time-consistency in an infinite-dimensional space. Although IBP formulae exist in such a framework (see~\cite[Lemma 9.34]{da2014stochastic}), the necessary conditions are not met here because of the degeneracy of the (finite-dimensional) noise.  
Even worse, the strong Feller property, a consequence of the IBP formula, is guaranteed to fail \cite[Section 3.2]{hamaguchi2023markovian}. 
The second representation is based on the Clark--Ocone formula and on an identity linking, at all times, the Malliavin derivative of the process, its tangent process (the derivative with respect to initial datum) and the inverse of the latter. As we saw in~\eqref{eq:flow_intro}, any dynamical analysis inevitably introduces the auxiliary process~$\widetilde{\bX}^{\sft,\bx}$. Hence such an identity is most likely out of reach for general stochastic Volterra processes. 

Instead of the Itô integral~$\int_0^T \eta_t\D W_t$, a general IBP method allows to express the weight as an (anticipative) Skorohod integral \cite[Proposition 6.2.1]{nualart2006malliavin}. This approach is favoured by~\cite{al2023computation} under rough volatility models. Unfortunately, checking that the conditions are met in practice is on uncharted territory (one has to study the inverse of the Malliavin derivative, or of the Malliavin matrix) and the Skorohod integral does not satisfy the requirement of being easily computable as it is not, in general, the limit of a sum of increments. We discuss this further in Section \ref{subsec:optimality}.

Let us conclude this quick overview of traditional techniques by saying that if the density exists and is differentiable then the weight $\pi$ can be written as the gradient of its logarithm \cite{broadie1996estimating}. Still, IBP formulae are often used to collect information on the density and not the other way around. Despite the recent study of the density regularity for Volterra processes~\cite{friesen2024regular}, too little is known at this stage outside of the trivial Gaussian case.

As we also wish to treat path-dependent test functions~$\phi$, let us mention that IBP formulae are derived under semimartingale models for various types of path-dependent payoffs, including Asian, barrier and lookback options~\cite{benhamou2001application,gobet2003computation, gobet2004revisiting}. These concern standard derivatives with respect to the initial condition. More recently, \cite{jazaerli2017functional} leverage the Dupire functional calculus to study the vertical derivative of the functional~$\Phi$ under an assumption of ``weak path dependence'', when $\bX^{\bx}$ is a semimartingale. This derivative consists in perturbing the path at the single point~$\sft$. The BEL formula they obtained morally corresponds to~\eqref{eq:IBP_intro} when~$K_b=K_\sigma=I_d$ and~$h$ is constant. 

Finally, results for SVEs are in short supply. Regarding SVEs with additive noise and constant initial datum, Fan proves IBP formulae for the standard derivative in \cite{fan2015integration} and~\cite{fan2015stochastic} for bounded and unbounded kernels, respectively. It is the method of the latter paper that we revisit and extend to functional derivatives in Proposition~\ref{prop:IBP_additive} below. Restraining further to a constant kernel~$K_b$, temporal coherence is somehow preserved and a BEL formula is obtained accordingly in~\cite{amine2018bismut}.


\subsection{Open questions}
The functional IBP formulae are witnesses of the smoothing property of the Volterra process' law. Through the lens of the fractional derivative, we make progress in our understanding of this regularisation, its increase as trajectories become rougher, and the trade-offs one can obtain between the test function's and the direction's regularities. 

This shows that a smoothing does hold in certain directions for the SVE's lift when the test functions are limited to the class~$f=\varphi\circ {\rm proj}$, where~${\rm proj}$ denotes the projection that maps the lift to the Volterra process and~$\varphi:\RR^d\to\RR$. Analogously to the strong Feller property, this could be a starting point for studying the lift's invariant measures and the SVE's large-time behaviour.

A complete understanding of derivatives in the kernel direction remains open; only the additive noise case is treated in Proposition~\ref{prop:IBP_additive}. In particular, this would contribute to a better comprehension of the Kolmogorov PDE derived in~\cite{gasteratos2025kolmogorov} and could help relax the regularity assumptions on the kernels and the test functions.

The question of existence and smoothness of a density is not addressed via this approach either. 
An objective for future research, which seems essential to make headway on this problem, is to get a better grasp at the inverse of the Malliavin matrix. The latter remains somewhat elusive in a Volterra context.

\subsection{Outline}
The rest of the paper is organised as follows. Section~\ref{sec:setup} gathers various results about the SVE~\eqref{eq:SVE_intro}, its tangent process, Malliavin derivative and the correspondence between them. It ends with an introduction to the weighted spaces we will use as path-dependent test functions. In Section~\ref{sec:IBP_1} can be found the functional BEL formula with directions in~$\Hh$ (Theorem~\ref{th:main_IBP}), the additive noise IBP formula with directions in $L^2(\Tt)$ (Proposition~\ref{prop:IBP_additive}), and finally a discussion on these results' optimality. The fractional IBP with directions in $\Hh_\alpha$ (Theorem~\ref{thm:frac_IBP}) is presented along with its proof and corollaries in Section~\ref{sec:IBP_frac}. Then we state and prove the second order IBP formula (Proposition~\ref{prop:IBP_2}) in Section~\ref{sec:IBP_2}. We display the application to stochastic volatility models in Section~\ref{sec:SV} and round up a couple of proofs in Appendix~\ref{sec:appendix}. 


\section{Setup and preliminary results}\label{sec:setup}

\subsection{Notations}

The notation $\lesssim$ is used to denote inequality of real numbers up to unimportant constants. The absolute value, Euclidean norm in~$\RR^d$ and Frobenius norm in~$\RR^{d\times d}$ will all be denoted by~$\abs{\cdot}$. The inner product in $\RR^d$ will be denoted~$\langle\cdot,\cdot\rangle$.  The gradient of a function $f:\RR^d\rightarrow \RR$ will be denoted by $\nabla f.$
Let $0\le \sft <T<+\infty$ and $\Tt:=[\sft,T]$. 

For Banach spaces~$E_1, E_2$, let $\Cc(E_1; E_2)$, $\Cc_b(E_1; E_2)$ and~$\mathscr{L}(E_1; E_2)$ be the space of continuous functions, bounded continuous functions and linear functions from~$E_1$ to $ E_2$, respectively. 
We define~$\Xx:=\Cc(\Tt;\RR^d)$ endowed with the topology of uniform convergence. We denote the supremum norm~$\norm{\cdot}_{\infty}$. For $\alpha\in(0,1)$, we define the space of H\"older continuous functions~$\Cc^\alpha:=\Cc^\alpha(\Tt;\RR^d)$ equipped with its norm
\begin{align*}
    \norm{f}_{\alpha} = \norm{f}_\infty + \sup_{\sft\le s\neq t\le T} \frac{\abs{f(t)-f(s)}}{\abs{t-s}^\alpha}.
\end{align*}
As we shall use it later, we also introduce~$\Cc^\alpha_\infty$ as the space of $\alpha$-H\"older continuous functions equipped with the supremum norm~$\norm{\cdot}_{\infty}$. 

 Next, for a Banach space~$F$, the Fréchet derivative of~$f:E_1\to E_2$ at $x\in E_1$, if it exists, is the unique operator~$Df(x)\in\mathscr{L}(F; E_2)$ such that
\begin{align}
    \lim_{\norm{h}_{F}\searrow 0} \frac{\norm{f(x+ h)-f(x)-Df(x)}_{ E_2}}{\norm{h}_{F}}=0.
\end{align}
Note that $f(x+h)$ is well-posed if, for instance, $E_1$ is an open subset of~$F$ or if~$F\subset E_1$. 
We say that $f\in \Dd^1_{F}(E_1; E_2)$ if it is Fréchet differentiable at any point $x\in E_1$ and in any direction~$h\in F$. If, moreover, $x\mapsto Df(x)(h)$ is continuous then we write~$f\in \Cc^1_{F}(E_1; E_2)$. 
We further denote $f\in \Dd_{b,F}^1(E_1; E_2)$ if $f\in \Dd_{F}^1(E_1; E_2)$ and there is a constant $C>0$ such that~$\abs{f(x)}_{ E_2}\le C,\,\abs{Df(x)h}_{ E_2}\le C \norm{h}_{F}$ for all~$x\in E_1$, and analogously for~$f\in\Cc_{b,F}^1(E_1; E_2)$. In this case, as a linear and bounded operator, $Df(x)$ is also continuous. Moreover, denote by $\Cc^1_{b,F,{\rm Lip}}(E_1;E_2)$ the subspace of $\Cc^1_{b,F}(E_1;E_2)$ functions with Lipschitz continuous derivatives, i.e. $\abs{Df(x)h-Df(y)h}_{E_2}\le C \norm{x-y}_{E_2}\norm{h}_{F}$ for some constant $C>0$. When $E_1, E_2$ are finite-dimensional we omit the subscript~$F$. For all these spaces we will omit to write~$E_2$ if it is~$\RR$.

In a similar fashion, we define the second order Fréchet derivative of~$f:E_1\to E_2$ at~$x\in E_1$ as the unique operator~$D^2f(x)\in\mathscr{L}(F,\mathscr{L}(F; E_2))$ such that~$D^2f(x)(g,h)=D(Df(\cdot)(g))(x)(h)$, for~$g,h\in F$. The spaces~$\Cc^2_{F}(E_1; E_2),\Dd^2_{F}(E_1; E_2)$ are defined analogously to the previous section. We say that~$f\in\Dd^2_{b,F}$ if~$f\in\Dd^1_{b,F}\cap\Dd^2_{F}$ and~$\abs{D^2f(x)(g,h)}\le C \norm{h}_{F}\norm{g}_{F}$ for all~$x\in E_1$.

When $E_1$ and $F$ are spaces of $\RR^d$-valued functions and $E_2=\RR$ we will be using the same notations with $\RR^{d\times d'}$-valued directions~$ h_1,h_2$. In that context, they mean
    \begin{align*}
        Df(x)(h_1)=\big(Df(x)(h_1 e_j)\big)_{j=1}^{d'} \quad \text{and}\quad 
        D^2f(x)(h_1,h_2)=\big(D^2f(x)(h_1e_i,h_2e_j)\big)_{i,j=1}^{d'},
    \end{align*}
    such that for any~$u,v,w\in\RR^{d'}$ we have
    \begin{equation}\label{eq:AbuseofNotations}
    Df(x)(h_1 u)=\sum_{j=1}^{d'} Df(x)(h_1 e_j)u_j,
        \quad D^2f(x)(h_1v,h_2 w)
        = \sum_{i,j=1}^{d'} D^2f(x)(h_1 e_i,h_2 e_j) v_i w_j. 
    \end{equation}
    
Throughout this work we fix a complete filtered probability space $(\Omega, \Ff, \{\Ff_t\}_{t\geq 0}, \PP)$ that supports an $m-$dimensional standard Brownian motion $W.$ Adopting notations and definitions from~\cite[Section~1.2]{nualart2006malliavin}, we denote by~$\Df$ the  Malliavin derivative operator with respect to~$W$ and by~$\DD^{1,2}$ its domain of application in~$L^2(\Omega)$. 
The derivative of a random variable $F\in\DD^{1,2}$ is a stochastic process in~$L^2(\Tt\times\Omega)$ denoted~$(\Df_t F)_{t\in\Tt}$.

We need two more notions of derivatives: the left and right Riemann--Liouville fractional derivatives. Letting $\alpha\in(0,1)$, $a,b\in\Tt$ and~$f\in L^1(\Tt;\RR)$, they are respectively defined as
\begin{alignat}{2}
    &\scrD^\alpha_{a^+} f(t) := \frac{\D}{\D t} \scrI^{1-\alpha}_{a^+} f(t), \qquad &&\scrI^\alpha_{a^+} f(t):=\frac{1}{\Gamma(\alpha)} \int_a^t (t-s)^{\alpha-1} f(s)\ds \label{eq:def_fracD_left}\\
    &\scrD^\alpha_{b^-} f(t) := - \frac{\D}{\D t} \scrI^{1-\alpha}_{b^-} f(t), \qquad &&\scrI^\alpha_{b^-} f(t):=\frac{1}{\Gamma(\alpha)} \int_t^b (s-t)^{\alpha-1} f(s)\ds. \label{eq:def_fracD_right}
\end{alignat}
These operators satisfy~$\scrD^\alpha_{a^+} \scrI^\alpha_{a^+} f=\scrD^\alpha_{b^-} \scrI^\alpha_{b^-} f=f$.

\subsection{Stochastic Volterra Equations}
Let $d\in\NN$, $\bx\in\Xx$,  $K_b,K_\sigma:\RR_+^2\to\RR^{d\times d},\,b:\RR_+\times\RR^d\to\RR^d$ and~$\sigma:\RR_+\times\RR^d\to \mathscr{L}(\RR^m,\RR^d)$. We consider the $\RR^d$-valued SVE
\begin{align}\label{eq:main_SVE}
    \bX_t^{\bx} =\bx(t) + \int_\sft^t K_b(t,s)b(s,\bX^{\bx}_s)\ds + \int_\sft^t K_\sigma(t,s)\sigma(s,\bX^{\bx}_s) \D W_s.
\end{align}
Let~$\Kk_0$ be the class of kernels~$K:\Tt^2\to\RR^{d\times d}$, borrowed from~\cite{gripenberg1980resolvents,zhang2010stochastic}, such that $K(t,s)=0$ for any~$s>t$,
\begin{align*}
    t\mapsto \int_\sft^t K(t,s)\ds \in L^\infty(\Tt), \quad \text{and} \quad \limsup_{\ep\to0} \norm{\int_{\cdot}^{\cdot+\ep}K(\cdot+\ep,s)\ds}_{\infty} =0.
\end{align*}
As the prototypical one-dimensional example in this work, we define the power-law kernel
\begin{equation}
    K_H(t,s):=(t-s)^{H-\half}\one_{t>s},\quad \text{for all}\quad H\in(0,1).
\end{equation}
In particular, $\Kk_0$ is a linear space and includes the kernels~$\KH$ and~$\KH^2$ for all~$H\in(0,1)$. Each~$K\in\Kk_0$ has a resolvent (see for instance~\cite[Lemma 2.1]{zhang2010stochastic}), that is a map~$R:\Tt^2\to\RR^{d\times d}$ such that
\begin{align}\label{eq:resolvent_pty}
        R(t,s)-K(t,s)=\int_s^t K(t,u)R(u,s)\du,\quad\text{for all } s<t, 
    \end{align}
and $\sup_{t\in\Tt}\big\lvert\int_\sft^t R(t,s)\ds\big\lvert<\infty$. For a constant~$C>0$, the resolvent~$R_{C,H}$ of~$K_H^2$ is given by 
\begin{equation}\label{eq:Res_powerlaw}
    R_{C,H}(t,s)=C\Gamma(2H)(t-s)^{2H-1} E_{2H,2H}\left(-C\Gamma(2H)(t-s)^{2H}\right)
\end{equation}
where~$\Gamma(z)$ is the Gamma function and $E_{\alpha,\beta}(z)= \sum_{n=0}^\infty \frac{z^n}{\Gamma(\alpha n+\beta)}$ denotes the Mittag--Leffler function. 
We take this opportunity to recall the generalised Grönwall lemma~\cite[Theorem 9.8.2]{Gripenberg90} suited to our purpose, which is the cornerstone of the well-posedness theory for~\eqref{eq:main_SVE} and of the further estimates in this section.
\begin{lemma}\label{lemma:Grownall}
    Let~$K\in\Kk_0$ and~$f,g:\RR_+\to\RR_+$ be two measurable functions such that~$t\mapsto\int_\sft^t K(t,s)f(s)\ds\in L^\infty(\Tt)$. Assume that for all~$t\in\Tt$, $\displaystyle
        f(t) \le g(t) + \int_\sft^t K(t,s)f(s)\ds,$ then:
    \begin{enumerate}\setlength{\itemsep}{-4pt}
        \item[1)] For all~$t\in\Tt$,
    \begin{align}\label{eq:Gronwall}
        f(t) \le g(t) + \int_\sft^t R(t,s)g(s)\ds.
    \end{align}
        \item[2)] If $K$ is positive then for all~$t\in\Tt$,
    \begin{align}\label{eq:RleK}
        \int_\sft^t R(t,s) g(s) \ds 
    \lesssim \sup_{t\in\Tt} \int_\sft^t K(t,s) g(s)\ds.
    \end{align}
        \item[3)] If $g(t)=\int_\sft^t K(t,s) g^*(s)\ds$ where~$g^*:\RR_+\to\RR_+$ is measurable, then for all~$t\in\Tt$,
        \begin{align*}
            f(t) \le \int_\sft^t R(t,s)g^*(s)\ds.
        \end{align*}
    \end{enumerate}
\end{lemma}
\begin{proof}
\begin{enumerate}
    \item[1)]  The first statement can be found in~\cite[Theorem 9.8.2]{Gripenberg90}.
    \item[2)]  For the second, note from the construction of the kernel~\cite[Equation (2.1)]{zhang2010stochastic} that~$K$ being positive implies that~$R$ is also positive. Thus, using \eqref{eq:resolvent_pty} and Fubini theorem we get
\begin{align*}
    \int_\sft^t R(t,s) g(s) \ds 
    &= \int_\sft^t K(t,s) g(s)\ds + \int_\sft^t \left(\int_s^t R(t,u)K(u,s)\du \right)g(s)\ds\\
    &=\int_\sft^t K(t,s) g(s)\ds + \int_\sft^t R(t,u) \left(\int_\sft^u K(u,s) g(s) \ds\right)\du \\
    &\le \left(1+\int_\sft^t R(t,u)\du\right) \sup_{t\in\Tt} \int_\sft^t K(t,s) g(s)\ds.
\end{align*}
    \item[3)] Lastly, we simply notice by Fubini's theorem that
\begin{align*}
    \int_\sft^t R(t,s)g(s)\ds 
    = \int_\sft^t \int_u^t R(t,s)K(s,u)\ds g^*(u) \du
    = \int_\sft^t (R-K)(t,u) g^*(u) \du.
\end{align*}
    Applying \eqref{eq:Gronwall}  then yields the claim. \qedhere
\end{enumerate}   
\end{proof}
The following assumptions are rather standard and will be in place throughout the paper.
\begin{assumptionp}
{\imperialBrick{\emph{Kernels}}}\label{as:kernel}  There exists~$H\in(0,1)$ such that for all~$\sft\le s<t\le t'\le T$
        \begin{align*}
            &\abs{K_b(t,s)}+\abs{K_\sigma(t,s)}^2 \lesssim \KH(t,s)^2; \\ 
            &\abs{K_b(t,s)-K_b(t',s)}+\abs{K_\sigma(t,s)-K_\sigma(t',s)}^2 \lesssim \abs{K_H(t,s)-K_H(t',s)}^2.
        \end{align*}
\end{assumptionp}
\begin{assumptionp}{\imperialBrick{\emph{Coefficients}}}\label{as:coefs}\
    \begin{enumerate}
        \item[(i)] For all $j\in\{1,\cdots,m\},$ $b,\sigma^{(j)}\in \Cc^{0,1}_b(\Tt\times\RR^d;\RR^d)$.
        \item[(ii)] For all~$(t,x)\in\Tt\times\RR^d$, the diffusion matrix~$\sigma(t,x)\in\RR^{d\times m}$ has a right inverse~$\xi(t,x):=\sigma(t,x)^{-1}\in\RR^{m\times d}$  such that~$\norm{\xi}_\infty=\sup_{(t,x)\in\Tt\times\RR^d} \abs{\xi(t,x)}<\infty$.
    \end{enumerate}
\end{assumptionp}
This set of assumptions implies the existence and uniqueness of a strong solution to the SVE~\eqref{eq:main_SVE} with all moments bounded:~$\sup_{t\in\Tt} \EE\abs{\bX^{\bx}_t}^p\lesssim \sup_{t\in\Tt} \abs{\bx(t)}^p$ for all $p\ge1$, see for instance~\cite[Theorem 3.1]{zhang2010stochastic}. Furthermore, $\bX^{\bx}$ is continuous and belongs to~$L^p(\Omega;\Xx)$ for all~$p\ge1$; if~$\bx\in\Cc^\gamma$ for some~$\gamma\in(0,1)$ then there exists a H\"older continuous modification~$\bX\in L^p(\Omega;\Cc^\gamma)$ for all~$p>1$. Both these results unfold from \cite[Theorem 3.3]{zhang2010stochastic}. We note in particular that $K_H$ satisfies \textbf{(H4)} of that paper because
    \begin{align*}
        \int_0^t \abs{K_H(t,s)-K_H(t',s)}^2\ds 
        \lesssim \int_0^t \Big( (t-s)^{2H-1} - (t'-s)^{2H-1}\Big) \ds 
        \lesssim (t'-t)^{2H}.
    \end{align*}

\subsection{The variational processes}\label{sec:tangent}
This sections deals with variational derivatives of~$\bX^{\bx}$, namely its derivative with respect to initial datum~$\bx$ and its Malliavin derivative. We could not locate such results in the literature under our set of assumptions but, since they are not surprising, we postpone the proofs to Appendix~\ref{sec:appendix}.

For all~$h\in\Xx$ we define the~$\RR^d$-valued process $\bY^{h}$ as the unique solution to the following linear SVE for all~$t\in\Tt$: 
\begin{align}\label{eq:tangent}
\bY^{h}_t = h(t) +\int_{\sft}^t K_b(t,s) \nabla b(s,\bX_s^{\bx}) \bY^{h}_s\ds + \sum_{j=1}^m \int_{\sft}^t K_\sigma(t,s) \nabla \sigma^{(j)}(s,\bX_s^{\bx}) \bY^{h}_s \D W_s^{(j)}.
\end{align}
Note that~$\int_{\sft}^t \KH(t,s)^2  \abs{h(s)}^p \ds \le \norm{h}_{\infty,r}^p T^{2H}/(2H),$ 
and~$\nabla b,\nabla \sigma$ are bounded hence~\cite[Theorem 3.1]{zhang2010stochastic} guarantees that the linear SVE~\eqref{eq:tangent} has a unique solution with all moments bounded. Furthermore, \cite[Theorem 3.3]{zhang2010stochastic} entails that~$\bY^{h}\in L^p(\Omega;\Xx)$ and~$\bY^{h}-h$ has a H\"older continuous modification on~$\Tt$. 
We call~$\bY^{h}$ the tangent process and we will show that it corresponds to the Fréchet derivative of~$\bX^{\bx}$ with respect to~$\bx$. 
In order to justify this claim, let us introduce the~$\RR^d$-valued process 
    \begin{align*}
        \bZ^{h}_t := \bX^{\bx+h}_t -\bX^{\bx}_t-\bY^{h}_t, \quad t\in\Tt.
    \end{align*}
The proof of the following result can be found in Appendix~\ref{subsec:proof_PropFrechet}.
\begin{proposition}\label{prop:Frechet}
Let Assumptions \ref{as:kernel} and~\ref{as:coefs} hold.
    Let~$h\in\Xx$. 
    Then the limit holds
    \begin{align*}
        \lim_{\norm{h}_{\infty} \!\!\!\searrow0} \frac{\EE\big\Vert\bZ^{h}\big\Vert_{\infty}^p}{\norm{h}_{\infty}^p} = 0.
    \end{align*}
    In other words, if we set $
    \Xx\ni\bx\mapsto f(\bx):=\bX^{\bx}\in L^p(\Omega;\Xx)$ then 
    we have $\bY^{h}=Df(\bx)(h)$. 
\end{proposition}


The Malliavin derivative satisfies a similar equation, which proof is in Appendix~\ref{subsec:proof_DX}.
\begin{lemma}\label{lemma:MalliavinD}
    For all~$t\ge0$, $\bX_t^{\bx}\in\DD^{1,2}$ and, for all $r\in[\sft,t)$, $\Df_r \bX^{\bx}$ is an $\RR^{d\times m}$-valued process solution of the linear equation 
    \begin{align}\label{eq:Malliavin_SVE}
    \mathcal{Y}_t = K_\sigma(t,r)\sigma(r,\bX_r^{\bx}) + \int_r^t K_b(t,s) \nabla b(s,\bX_s^{\bx}) \mathcal{Y}_s \ds+ \sum_{j=1}^m\int_r^t K_\sigma(t,s) \nabla \sigma^{(j)}(s,\bX_s^{\bx}) \mathcal{Y}_s \D W_s^{(j)},
\end{align}
and $\Df_r \bX_t^{\bx}=0$ if $r\ge t$. Moreover, there is $C>0$ such that $\EE\abs{\Df_r \bX_t^{\bx}}^2\le C K_H(t,r)^2$.
\end{lemma}
Finally, since the Malliavin derivative can be interpreted as a directional derivative of the Brownian motion along an $L^2(\Omega\times\Tt)$ path~$\eta$, the following lemma makes the link between the tangent process and Malliavin derivative in a neat fashion. A similar result can be traced back to~\cite{coutin2001stochastic}.
\begin{lemma}
    \label{lemma:equality_Y_DX}
    Let $h(t) = \int_{\sft}^t K_\sigma(t,s) h^*(s)\ds$ for some~$h^*\in L^2(\Tt;\RR^d)$.  
    For all~$t\in\Tt$, we have
    \begin{equation}
        \bY^{h}_t = \int_{\sft}^T \Df_s \bX_t^{\bx} \xi(s,\bX^{\bx}_s)h^*(s) \ds.
    \end{equation}
\end{lemma}
\begin{proof}
    For any $\eta\in L^2(\Omega\times\Tt;\RR^m)$, applying Fubini and stochastic Fubini theorems to~\eqref{eq:Malliavin_SVE} entail that the $\RR^d$-valued process $(\langle \Df \bX_t^{\bx}, \eta \rangle_{L^2(\Tt;\RR^m)})_{t\in\Tt}$ satisfies
    \begin{align*}
        \langle \Df \bX_t^{\bx}, \eta \rangle_{L^2(\Tt;\RR^m)}
        &= \int_{\sft}^T K_\sigma(t,s)\sigma(s,\bX^{\bx}_s) \eta_s \ds +  \int_{\sft}^t K_b(t,s) \nabla b(\bX^{\bx}_s) \left(\int_{\sft}^s \Df_r \bX_s^{\bx} \eta_r \dr\right) \ds\\
        &\qquad+ \sum_{j=1}^m\int_{\sft}^t K_\sigma(t,s) \nabla^{(j)} \sigma(\bX^{\bx}_s) \left(\int_{\sft}^s \Df_r \bX_s^{\bx} \eta_r \dr\right) \D W_s^{(j)}\\          
        &= \int_{\sft}^t K_\sigma(t,s)\sigma(s,\bX^{\bx}_s) \eta_s \ds + \int_{\sft}^t K_b(t,s) \nabla b(\bX^{\bx}_s) \langle \Df \bX_s^{\bx}, \eta \rangle_{L^2(\Tt;\RR^m)} \ds\\
        &\qquad+ \sum_{j=1}^m\int_{\sft}^t K_\sigma(t,s) \nabla^{(j)} \sigma(\bX^{\bx}_s) \langle \Df \bX_s^{\bx}, \eta \rangle_{L^2(\Tt;\RR^m)} \D W_s^{(j)},
    \end{align*}
    where we used, crucially, that~$K_\sigma(t,s)=0$ for~$s\ge t$ and~$ \Df_r \bX_s^{\bx}=0$ for $r\ge s$. 
    Setting~$h(t)=\int_{\sft}^t K_\sigma(t,s) h^*(s) \ds$ for some~$h^*\in L^2(\Tt;\RR^d)$ and $\eta_s = \xi(s,\bX^{\bx}_s)h^*(s)$, $\bY^{h}$ solves the same equation~\eqref{eq:tangent}. By pathwise uniqueness, the proof is complete. 
\end{proof}

\subsection{Weighted spaces}\label{subsec:weighted_space}
The proof strategy of the IBP formula consists in deriving it first for~$\Cc^1_{b,\Xx,{\rm Lip}}(\Xx)$ test functions and then extending by density to a space of less regular functions. We define for that purpose our targetted space of test functions, which is a weighted space inspired by the theory of generalized Feller processes~\cite{cuchiero2026ramifications}. We fix $\gamma\in(0,H)$ and $p\ge1$ (recall that~$\bX^{\bx}-\bx$ has~$\gamma$--H\"older trajectories almost surely for all~$\gamma\in(0,H)$). Recall that $\Cc^\gamma_\infty$ is the space of $\gamma$-H\"older continuous functions equipped with the supremum norm~$\norm{\cdot}_{\infty}$. 
We then define the weight function~$\varrho:\Cc^\gamma_\infty\to\RR_+$ as
\begin{equation}\label{eq:def_varrho}
\varrho(\bx):=1+\norm{\bx}_\gamma^p.
\end{equation}
As explained in \cite[Example 2.3 (iv)]{cuchiero2026global}, this guarantees that for any $R>0$ the sets $B_R:=\{\bx\in\Cc^\gamma_\infty: \varrho(\bx) \le R\}$ are compact subsets of $\Cc^\gamma_\infty$. Further define the weighted norm
\begin{align*}
    \norm{\phi}_{\varrho}:=\sup_{\bx\in\Cc^\gamma_\infty} \frac{\abs{\phi(\bx)}}{\varrho(\bx)}, \quad \text{for any  } \phi:\Cc^\gamma_\infty\to\RR.
\end{align*}
The space~$\Bb^{\varrho} (\Cc^\gamma_\infty)$ is defined as the closure of~$ \Cc_b(\Cc^\gamma_\infty)$ with respect to~$\norm{\cdot}_{\varrho}$~\cite[Definition 2.3]{cuchiero2020generalized}, and is also a Banach space when equipped with this norm. 
A useful characterisation, proven in  \cite[Theorem 2.2]{cuchiero2026ramifications}, states that $\phi\in\Bb^{\varrho}(\Cc^\gamma_\infty)$ if and only if, for all~$R>0$, $\varphi\lvert_{B_R}\in  \Cc_b(B_R)$ and \begin{equation}\label{eq:VanishInfty}
    \lim_{R\nearrow+\infty}\sup_{\bx\in\Xx\setminus B_R}\frac{\abs{\phi(\bx)}}{\varrho(\bx)}=0.
\end{equation} 
In particular, $\Bb^{\varrho}(\Cc^\gamma_\infty)$ includes all continuous functions with the controlled growth~$\abs{\phi(\bx)}\le C(1+\norm{\bx}^p_\gamma)$ for some $C>0$. We thus have the following density lemma.
\begin{lemma}\label{lemma:density_phi}
Let $\phi\in\Bb^{\varrho}(\Cc^\gamma_\infty)$.
\begin{enumerate}
    \item[1)] For any $N>0$, $\bx\in\Cc^\gamma_\infty$ it holds
    \begin{align}\label{eq:est_phiX}
        \EE\abs{\phi(\bX^{\bx})}^2\lesssim N^2\norm{\phi}_{\varrho}^2 + \frac{1+\norm{\bx}_{\gamma}^{5p/2}}{\sqrt{N}}.
    \end{align}
    \item[2)] There exists $(\phi_n)\subset \Cc^1_{b,\Cc^\gamma,{\rm Lip}}(\Cc^\gamma_\infty)$ such that~$\lim_{n\nearrow+\infty}\norm{\phi-\phi_n}_{\varrho}=0$. In particular, \begin{align*}
        \lim_{n\nearrow+\infty}\sup_{\bx\in B_R}\EE\abs{\phi_n(\bX^{\bx})-\phi(\bX^{\bx})}^2=0.
    \end{align*}
\end{enumerate}
\end{lemma}

\begin{proof}
1) For any~$\phi\in\Bb^{\varrho}(\Cc^\gamma_\infty)$ we have
\begin{align*}
    \sup_{\bx\in B_N} \abs{\phi(\bx)}
    = \sup_{\bx\in B_N} \abs{\phi(\bx)} \frac{{\varrho}(\bx)}{{\varrho}(\bx)}
    \le N\norm{\phi}_{{\varrho}}.
\end{align*} 
On the other hand, Cauchy-Schwarz and Markov inequalities with $\phi(\bx)\le\norm{\phi}_{\varrho}\varrho(\bx)$ yield
\begin{align*}
    \EE\left[\abs{\phi(\bX^{\bx})}^2\one_{\bX^{\bx}\notin B_N}\right]^2
    &\le \norm{\phi}^4_{\varrho}\EE\big[(1+\norm{\bX^{\bx}}^{p}_{\gamma})^4\big] \, \PP\big(1+\norm{\bX^{\bx}}_{\gamma}^p\ge N\big) \\
    &\lesssim  \Big(1+\EE\norm{\bX^{\bx}}^{4p}_{\gamma} \Big) \frac{1+\EE\norm{\bX^{\bx}}_{\gamma}^p}{N}
    \lesssim \frac{1+\norm{\bx}_{\gamma}^{5p/2}}{\sqrt{N}}.
\end{align*}
Writing $
\EE\abs{\phi(\bX^{\bx})}^2
=\EE\Big[\abs{\phi(\bX^{\bx})}^2(\one_{\bX^{\bx}\in B_N} + \one_{\bX^{\bx}\notin B_N})\Big]$ and applying the two estimates above yields the claim.

2) We first prove the existence of an approximating limit $\phi_n$. 
    We invoke the Stone-Weierstrass theorem \cite[Proposition 2.6]{cuchiero2026ramifications} that holds on non-locally compact spaces. 
Clearly $\Cc^1_{b,\Cc^\gamma,{\rm Lip}}(\Cc^\gamma_\infty)$ is a subalgebra of $\Bb^{\varrho}(\Cc^\gamma_\infty)$ with respect to pointwise multiplication and contains constant functions.  
Moreover it separates points: for any~$\bx\neq\bm{y}\in\Cc^\gamma_\infty$ there is an interval~$I\subset\Tt$ such that~$\max_{t\in\Tt}\abs{\bx(t)-\bm{y}(t)}=\max_{t\in I}\abs{\bx(t)-\bm{y}(t)}$, and either $\bx(t)>\bm{y}(t)$ for all $t\in I$ or $\bx(t)>\bm{y}(t)$ for all $t\in I$. Then the function~$f(\bx)=\int_I\bx(t)\dt$ separates points, is continuous and differentiable as~$Df(\bx)(h)=\int_I h(t)\dt$ for any~$h\in\Cc^\gamma$. Moreover, the derivatives are clearly Lipschitz continuous and bounded since $|Df(\bx)(h)|\lesssim \norm{h}_\infty\lesssim \norm{h}_{\gamma}$. The Stone-Weierstrass theorem implies that~$\Cc^1_{b,\Cc^\gamma,{\rm Lip}}(\Cc^\gamma_\infty)$ is dense in~$\Bb^{\varrho}(\Cc^\gamma_\infty)$ with respect to~$\norm{\cdot}_{\varrho}$.

Let us now apply~\eqref{eq:est_phiX} to $\phi-\phi_n$ and to an arbitrary radius $N>0$. Noting that $\bx\in B_{R}$ implies $1+\norm{\bx}_\gamma^{5p/2}\le 1+(R-1)^{5/2}\le R^{5/2}$, we get
\begin{align*}
    \sup_{\bx\in B_R}\EE\abs{\phi_n(\bX^{\bx})-\phi(\bX^{\bx})}^2
    \lesssim N^2 \norm{\phi-\phi_n}^2_{\varrho} + \frac{R^{5/2}}{\sqrt{N}}.
\end{align*}
Choosing $N$ large enough and then passing to the limit as $n$ goes to $+\infty$ yields the claim.
\end{proof}


\section{Bismut--Elworthy--Li formulae}\label{sec:IBP_1}

\subsection{The path-dependent Bismut--Elworthy--Li formula}
The specificity of the first Bismut--Elworthy--Li formula presented in this section is that the class of admissible direction is restricted to the Cameron--Martin space~$\Hh$ associated to the kernel~$K_\sigma$. Let us define~$K_\sigma^*f(t):=\int_\sft^tK_\sigma(t,s)f(s)\ds$ and
\begin{align*}
    \Hh := K_\sigma^*(L^2(\Tt))= \bigg\{ h\in\Xx: h(t) = \int_\sft^t K_\sigma(t,s) h^*(s) \ds \text{    for some    } h^*\in L^2(\Tt;\RR^m)\bigg\}.
\end{align*}
This is a Hilbert space when endowed with the inner product~$\langle h,g\rangle_\Hh = \langle h^*,g^*\rangle_{L^2(\Tt)}$. It also corresponds to the class of directions for which Lemma~\ref{lemma:equality_Y_DX} holds. Moreover this space is included in $\Cc^H(\Tt)$ since for all~$h\in\Hh$ it holds, by Assumption~\ref{as:kernel} and Cauchy--Schwarz inequality,
\begin{align}\label{eq:HhinCH}
    \abs{h(t)-h(s)}\le \norm{h^*}_{L^2(\Tt)} \Big(\abs{K_H(t,\cdot)-K_H(s,\cdot)}_{L^2([0,s])}+ \abs{K_H(t,\cdot)}^2_{L^2([s,t])}\Big)\lesssim \norm{h}_{\Hh} (t-s)^H.
\end{align}
In the special but central cases where~$K_\sigma\equiv K_H$ or $K_\sigma$ is the kernel of the Mandelbrot--Van Ness fractional Brownian motion, $\Hh$ is nothing else than~$\mathscr{I}_{\sft^+}^{H+1/2}(L^2(\Tt))$ (albeit with different but equivalent norms) which is dense in the space of continuous functions started at zero, see e.g. \cite{picard2010representation}.  Different characterisations of this space can be found in \cite[Section 13.2]{samko1993fractional}. To give an idea let us simply mention that $\Hh$ is included in the space of $H$-H\"older continuous functions starting at zero (at $\sft$). When $K_\sigma=K_H$ we obtain $h^* = \mathscr{D}^{H+1/2}_{\sft^+}(h) / \Gamma(H+1/2)$ (defined in~\eqref{eq:def_fracD_left} hence the power-law $h(t)=(t-\sft)^{H+\ep}$ is in $\Hh$ for all~$\ep>0$ because $h^*(t) =c (t-\sft)^{\ep-1/2}$.

We are now ready to state the main result of this section. Its proof unfolds from the identity $\bY^{h}_t = \langle \Df \bX^{\bx}_t, \eta \rangle_{L^2(\Tt;\RR^m)}$ at all~$t\in\Tt$, where $h\in\Hh$ and where we identify $\eta_s = \xi(s,\bX^{\bx}_s)h^*(s)$, as stated in Lemma~\ref{lemma:equality_Y_DX}.
\begin{theorem}\label{th:main_IBP}
    Let Assumptions~\ref{as:kernel} and \ref{as:coefs} hold, $\gamma\in(0,H)$ and $\phi\in \Bb^{\varrho}(\Cc^\gamma_\infty)$. Then the mapping $ \bx\mapsto\Phi(\bx):=\EE[\phi(\bX^{\bx})]$ belongs to~$\Dd^1_{\Hh}(\Cc^\gamma_\infty)$ and for all~$\bx\in\Cc^\gamma,\,h\in\Hh$, the Bismut--Elworthy--Li formula holds 
    \begin{equation}\label{eq:main_IBP} 
        D \Phi(\bx)(h)
        = \EE\left[\phi(\bX^{\bx})\int_\sft^T \Big\langle\xi(s,\bX_s^{\bx})h^*(s), \D W_s\Big\rangle\right].
    \end{equation}
\end{theorem}
\begin{remark}\label{rem:mainIBP}
Under stronger assumptions we can break free from the H\"older regularity:
\begin{enumerate}
    \item[i)] The first part of the proof goes through with $\phi\in\Cc^1_{b,\Hh,{\rm Lip}}(\Xx)$, entailing that $\Phi\in\Dd^1_{\Hh}(\Xx)$ and that the integration by parts formula holds for any~$\bx\in\Xx$.
    \item[ii)] Under the condition that~$\phi(\bx)=\varphi(\bx(T))$ for all~$\varphi:\RR^d\to\RR$ with polynomial growth, it is also possible to prove by density that~$\Phi\in \Dd^1_{\Hh}(\Xx)$ and to derive the IBP for all~$\bx\in\Xx$. The reason is that, since $\RR^d$ is locally compact, it is no longer necessary to appeal to the H\"older norm in Lemma~\ref{lemma:density_phi}.
\end{enumerate}
Moreover, unlike classical Bismut--Elworthy--Li formulae, the tangent process $\bm{Y}^h$ does not show up in the final form. Hence the right side of \eqref{eq:main_IBP} is well-defined as long as the SVE has a unique solution~$\bX^{\bx}$ and we can hope to relax the conditions~$b,\sigma\in\Cc^{0,1}_b$. 
\end{remark}

\begin{proof}
\emph{\underline{Step 1. Smooth test functions.}}
Let us assume here that $\phi\in\Cc^1_{b,\Cc^\gamma,{\rm Lip}}(\Cc^\gamma_\infty)$.
The integration by parts follows from the following computation, where we differentiate under the expectation, we apply the chain rule from \cite[Proposition 3.8]{pronk2014tools}, Lemma~\ref{lemma:equality_Y_DX}, the linearity of the Fréchet derivative and the Malliavin integration by parts~\cite{nualart2006malliavin} to conclude 
\begin{equation} \label{eq:IBP_computation}
\begin{aligned}
    D \Phi(\bx)(h) 
    = \EE\Big[ D \phi(\bX^{\bx}) (\bY^{h}) \Big] 
    &= \EE\left[\int_\sft^T D \phi(\bX^{\bx})( \Df_r \bX^{\bx}) \xi(s,\bX^{\bx}_s) h^*(s) \ds \right]  \\
    &= \EE\left[\int_\sft^T \Big\langle\Df_s \phi(\bX^{\bx}),\xi(s,\bX_s^{\bx}) h^*(s) \Big\rangle \ds \right] \\
    &= \EE\left[\phi(\bX^{\bx}) \int_\sft^T \Big\langle\xi(s,\bX_s^{\bx}) h^*(s), \D W_s \Big\rangle \right]. 
\end{aligned}
\end{equation}
We justify the swapping of differentiation and expectation by the regularity of $\phi\in \Cc^1_{b,\Hh,{\rm Lip}}(\Cc^\gamma_\infty)$ and the chain rule for Fréchet derivatives. Indeed, they entail
\begin{align*}
    \EE\big[\abs{\phi(\bX^{\bx+h})-\phi(\bX) - D\phi(\bX^{\bx})(\bY^{h})}\big] 
    &\le \int_0^1 \EE\abs{D\phi(\bX^{\bx+\lambda h})(\bY^{h})-D\phi(\bX^{\bx})(\bY^{h})}\D \lambda\\
    &\lesssim \int_0^1 \lambda \norm{h}_{\gamma} \D \lambda \; \EE\norm{\bY^{h}}_{\gamma}
    \lesssim \norm{h}_{\Hh},
\end{align*}
where we used~\eqref{eq:HhinCH} to conclude.
This yields the claim for smooth enough~$\phi$, more precisely we showed that~$\Phi\in\Dd^1_{b,\Hh}(\Cc^\gamma_\infty)$.

\emph{\underline{Step 2. General test functions.}} Recall that under Assumption~\ref{as:kernel} we have  $\Hh\subset\Cc^H\subset\Cc^\gamma$, see~\eqref{eq:HhinCH}. Let us call $(\phi_n)$ the sequence of $\Cc^1_{b,\Hh,{\rm Lip}}(\Cc^\gamma_\infty)$ functions constructed in Lemma~\ref{lemma:density_phi}, and for all~$\bx\in\Cc^\gamma_\infty,\,h\in\Hh$, define
\begin{align*}
    \Phi_n(\bx)(h):=\EE[\phi_n(\bX^{\bx})],\qquad 
    \Psi_1(\bx;h):=\EE\left[\phi(\bX^{\bx}) \int_\sft^T \Big\langle \xi(s,\bX_s^{\bx})h^*(s) \D W_s \Big\rangle \right].
\end{align*}
For any $m,n\in\NN$, Taylor's formula, the IBP~\eqref{eq:IBP_computation} and Cauchy--Schwarz inequality yield
\begin{equation}\label{eq:est_PhinPhi_m}
\begin{aligned}
    \abs{(\Phi_n-\Phi_m)(\bx+h) - (\Phi_n-\Phi_m)(\bx)} 
    &=\abs{\int_0^1 D(\Phi_n-\Phi_m)(\bx+\lambda h)(h)\D \lambda}\\
    &=\abs{\int_0^1\EE\left[\big(\phi_n-\phi_m\big)(\bX^{\bx+\lambda h})\int_\sft^T \xi(s,\bX_s^{\bx+\lambda h})h^*(s)\D W_s\right]\D \lambda} \\
    &\le \sup_{\lambda\in[0,1]}\norm{\big(\phi_n-\phi_m\big)(\bX^{\bx+\lambda h})}_{L^2(\Omega)} \norm{\xi}_{\infty} \norm{h}_{\Hh},
\end{aligned}
\end{equation}
where we recall that $\norm{h^*}_{L^2(\Tt)}=\norm{h}_{\Hh}$. Applying the estimate~\eqref{eq:est_phiX} to $\phi_n-\phi_m$ yields for any $N>1$ 
\begin{align}
    \abs{(\Phi_n-\Phi_m)(\bx+h) - (\Phi_n-\Phi_m)(\bx)}
    \lesssim \left(N\norm{\phi_n-\phi_m}_{{\varrho}} + \frac{1+(\norm{\bx}_{\gamma}+\norm{h}_\gamma)^{5p/4}}{N^{1/4}}\right) \norm{h}_\Hh.
\end{align}
Recall that~$\norm{h}_\gamma\le C \norm{h}_{\Hh}$ by \eqref{eq:HhinCH}. 
Let us fix $\ep>0$, taking the limit as $m$ to $+\infty$ and $N$ large enough, there is $n_1\in\NN$ such that for all $n\ge n_1$, 
\begin{align}
   \abs{(\Phi_n-\Phi)(\bx+h) - (\Phi_n-\Phi)(\bx)}
    \le \ep \norm{h}_\Hh.
\end{align}
In a similar fashion, the IBP \eqref{eq:IBP_computation} and Cauchy--Schwarz inequality yield the existence of $n_2\in\NN$ such that for all $n\ge n_2$:
\begin{align}\label{eq:cvg_IBP}
    \big\lvert D\Phi_n(\bx)(h) - \Psi_1(\bx;h) \big\lvert^2
    &\le  \EE\abs{\phi_n(\bX^{\bx})-\phi(\bX^{\bx})}^2 \EE\left[\int_\sft^T \abs{\xi(r,\bX_r^{\bx})}^2 \abs{h^*(r)}^2\dr\right] \\ 
    &\le \EE\abs{\phi_n(\bX^{\bx})-\phi(\bX^{\bx})}^2 \norm{\xi}_\infty^2 \norm{h}_{\Hh}^2
    \le \ep \norm{h}_{\Hh}. \nonumber
\end{align}
Therefore, for $n\ge n_1 \vee n_2$ we obtain
\begin{equation}\label{eq:frechet_density}
\begin{aligned}
    \frac{\lvert\Phi(\bx+h)-\Phi(\bx)-\Psi_1(\bx;h)\lvert}{\norm{h}_{\Hh}}
    &\le  \frac{\lvert(\Phi-\Phi_n)(\bx+h)-(\Phi-\Phi_n)(\bx)\lvert}{\norm{h}_\Hh} \\
    &\qquad+  \frac{\lvert D\Phi_n(\bx)h-\Psi_1(\bx;h)\lvert}{\norm{h}_{\Hh}}\\ 
    &\qquad+  \frac{\lvert\Phi_n(\bx+h)-\Phi_n(\bx)-D\Phi_n(\bx)h\lvert}{\norm{h}_{\Hh}} \\
    &\le 2\ep +  \frac{\lvert\Phi_n(\bx+h)-\Phi_n(\bx)-D\Phi_n(\bx)h\lvert}{\norm{h}_{\Hh}} ,
\end{aligned}
\end{equation}
and since $\Phi_n\in\Dd^1_{\Hh}(\Cc^\gamma_\infty)$ for all $n\in\NN$, for $\norm{h}_{\Hh}$ small enough the above is lower or equal than $3\ep$.
 This proves that~$D\Phi(\bx)(h)=\Psi_1(\bx;h)$. Cauchy--Schwarz inequality yields
 \begin{align*}
     \abs{\Psi_1(\bx;h)} \lesssim \norm{\phi(\bX^{\bx})}_{L^2(\Omega)} \norm{h^*}_{L^2(\Tt)},
 \end{align*}
 and $\abs{\phi(\bX^{\bx})}\le C (1+\norm{\bX^{\bx}}^p_\gamma)$ because $\phi\in\Bb^{\varrho}(\Cc^\gamma_\infty)$, which yields the claim.
\end{proof}

\begin{remark}
    For any $\Ff_{\sft}$-measurable path~$\xi$, $\Phi(\xi)=\EE[\phi(\bX^\xi)\lvert\Ff_{\sft}]$ (see the proof of Lemma~\ref{lemma:DalphaM_is_Malpha} for a justification) hence the IBP~\eqref{eq:main_IBP} extends to conditional expectations as well. 
\end{remark}

\begin{remark}[Girsanov viewpoint]
    In retrospect, one could exploit Girsanov theorem to write a proof without Malliavin calculus. Let~$h\in\Hh$ and $W^{\sft}:=W-W_{\sft}$. First notice that there is a measurable~$F:\Xx\times\Cc(\Tt;\RR^m)\to\RR$ such that~$\phi(\bX^{\bx})=F(\bx,W^{\sft})$, and hence~$\phi(\bX^{\bx+h})=F(\bx,W^{\sft} + \int_{\sft}^\cdot \xi(s,\bX^{\bx}_s)h^*(s)\ds)$. We further obtain
    \begin{align*}
        \EE[\phi(\bX^{\bx+h})] 
        = \EE\left[F\left(\bx,W^{\sft} + \int_{\sft}^\cdot \xi(s,\bX^{\bx}_s)h^*(s)\ds\right)\right] = \EE\left[F(\bx,W^{\sft})\frac{\D \PP_h}{\D\PP}\right],
    \end{align*}
    where~$\PP_h$ is the probability measure under which~$W^{\sft} + \int_{\sft}^\cdot \xi(s,\bX^{\bx}_s)h^*(s)\ds$ is a Brownian motion. Differentiating~$\Phi(\bx)$ along~$h$ thus boils down to differentiating $\frac{\D \PP_h}{\D\PP}$ around $h=0$.
\end{remark}

The smoothing property of the expectation only applies in certain directions. 
In fact, the strong Feller property for Markovian lifts of Volterra processes was shown \emph{not} to hold in~\cite[Theorem 3.3]{hamaguchi2023markovian} and we do not expect such a strong property to hold for the Volterra processes either. 
\begin{corollary}\label{coro:reg_Phi}
Let the same Assumptions as Theorem \ref{th:main_IBP} hold.
    Then for all $\bx\in\Cc^\gamma$ and~$h\in\Hh$ we have
    \begin{align*}
        \abs{\Phi(\bx+h)-\Phi(\bx)}
        \lesssim \norm{h}_{\Hh}.
    \end{align*}
\end{corollary}
\begin{proof}
    Applying Cauchy-Schwarz inequality to the formula~\eqref{eq:main_IBP}, we verify that for any~$\bx\in\Cc^\gamma$
    \begin{align*}
        \abs{D\Phi(\bx)(h)} \le \left(\EE\abs{\phi(\bX^{\bx})}^2\right)^\half  \norm{\xi}_{\infty} \norm{h^*}_{L^2(\Tt)}
        =\norm{\phi(\bX^{\bx})}_{L^2(\Omega)} \norm{\xi}_{\infty}\norm{h}_{\Hh}.
    \end{align*}
    The fundamental theorem of integration yields
    \begin{align*}
        \abs{\Phi(\bx+h)-\Phi(\bx)}
        =\Big\lvert \int_0^1 D\Phi(\bx+\lambda h)(h)\D \lambda \Big\lvert
        \le \int_0^1\norm{\phi(\bX^{\bx+\lambda h})}_{L^2(\Omega)} \D \lambda\norm{\xi}_{\infty} \norm{h}_{\Hh},
    \end{align*}
    thereby concluding the proof.
\end{proof}
\subsection{A state-dependent Bismut--Elworthy--Li formula with additive noise}
The proof of Theorem~\ref{th:main_IBP} hinges on the equality~$\bY^{h}_t = \int_{\sft}^T \Df_s \bX^{\bx}_t \xi(s,\bX_s^{\bx}) h^*(s) \ds$ being true for all~$t\in\Tt$. In the state-dependent case~$\phi(\bx)=\varphi(\bx(T))$, it is sufficient to find a similar relation only at~$t=T$. One can achieve this by assuming that the noise is additive, thereby greatly relaxing the conditions on the admissible directions. This technique is inspired by \cite[Theorem 5]{fan2015stochastic}.
\begin{proposition}\label{prop:IBP_additive}
    Let Assumptions~\ref{as:kernel} and \ref{as:coefs} hold, $\sigma\equiv\sigma_0\in\RR^{d\times m}$ be constant, $K_b=K_\sigma=:K$ and $\varphi:\RR^d\to\RR$ be such that~$\varphi(\bX^{\bx}_T)\in L^p(\Omega)$ for some $p>2$. Furthermore assume that there exists $a\in L^2(\Tt;\RR^{d\times d})$ such that~$ \int_{\sft}^T K(T,s)a_s \ds=I_d$. Then the mapping $ \bx\mapsto\Phi(\bx):=\EE[\varphi(\bX^{\bx}_T)]$ belongs to~$\Dd^1_{\Xx}(\Xx)$ and for all~$\bx, h\in\Xx$ the IBP formula holds
    \begin{equation}\label{eq:additive_IBP}      
        D\Phi(\bx)(h) 
        = \EE\left[\varphi(\bX^{\bx}_T)\int_\sft^T \big\langle \eta_s^{a,h}, \D W_s\big\rangle\right],
    \end{equation}
    where   
    \begin{align*}
        \eta^{a,h}_s:=\sigma_0^{-1}\left[\nabla b(\bX^{\bx}_s) \left(h(s)-\int_{\sft}^s K(s,r)a_r\dr h(T)\right) + a_s h(T) \right].
    \end{align*}
    Finally, let $h\in L^2(\Tt;\RR^d)$, let~$(h^\delta)_{\delta>0}\subset\Xx$ be an approximating sequence in $L^2(\Tt;\RR^d)$ and define $h(T):=\lim_{\delta\searrow0}h^\delta(T)$. Then we have
    \begin{align*}
        \lim_{\delta\searrow0} D\Phi(\bx)(h^\delta) = \EE\left[\varphi(\bX^{\bx}_T) \int_{\sft}^T \big\langle \eta_s^{a,h},\D W_s\big\rangle\right].
    \end{align*}
    This holds in particular for unbounded directions such as~$h(t)=K_\sigma(t,\sft)$ and~$h^\delta(t)=K_\sigma(t\vee(\sft+\delta),\sft)$ or $h^\delta(t)=K_\sigma(t+\delta,\sft)$.
\end{proposition}
\begin{remark}
 For any $u\in L^2(\Tt;\RR^{d\times d})$ such that $\int_{\sft}^T K(T,r)u_r\dr$ is invertible, the conditions on $a$ are met by choosing
    \begin{equation}\label{eq:choice_a_1d}
    a_s=u_s\left(\int_{\sft}^T K(T,r)u_r\dr\right)^{-1}.
    \end{equation}
\end{remark}
\begin{proof}
    \emph{\underline{Step 1.}} Let $\varphi\in\Cc^1_b(\RR^d)$ and $\bx,h\in\Xx$. Notice from \eqref{eq:tangent} that $\bY^h$ solves
    \begin{align*}
        \bY^h_t = h(t) + \int_{\sft}^t K(t,s) \nabla b(\bX^{\bx}_s)\bY^h_s \ds,
    \end{align*}
    while, for any $\eta\in L^2(\Omega\times\Tt;\RR^m)$ and according to~\eqref{eq:Malliavin_SVE}, the $\RR^d$-valued process $(\langle \Df \bX_t^{\bx}, \eta \rangle_{L^2(\Tt;\RR^m)})_{t\in\Tt}$ satisfies
    \begin{align*}
        \langle \Df \bX_t^{\bx}, \eta \rangle_{L^2(\Tt;\RR^m)} = \int_{\sft}^t K(t,s)\sigma_0 \eta_s \ds + \int_{\sft}^t K(t,s) \nabla b(\bX^{\bx}_s) \langle \Df \bX_s^{\bx}, \eta \rangle_{L^2(\Tt;\RR^m)} \ds.
    \end{align*}
    Therefore, $\Delta_t^\eta:= \bY^h_t - \langle \Df \bX_t^{\bx}, \eta \rangle_{L^2(\Tt;\RR^m)}$ is the unique solution to
    \begin{align}\label{eq:Delta}
        \Delta_t^\eta = h(t) -\int_{\sft}^t K(t,s)\sigma_0 \eta_s \ds + \int_{\sft}^t K(t,s) \nabla b(\bX^{\bx}_s)\Delta_s^\eta \ds.
    \end{align}
    Let us pick $a\in L^2(\Tt;\RR^{d\times d})$ such that~$\int_{\sft}^T K(T,s)a_s \ds=I_d$ and define
    \begin{align*}
        \eta^{a,h}_s:=\sigma_0^{-1}\left[\nabla b(\bX^{\bx}_s) \left(h(s)-\int_{\sft}^s K(s,r)a_r\dr h(T)\right) + a_s h(T) \right].
    \end{align*} 
    Plugging this into~\eqref{eq:Delta}, we find that~$\Delta_t^{\eta^{a,h}} = h(t) -\int_{\sft}^t K(t,s)a_s\ds h(T)$ is the unique solution to this equation, and thus~$\bY^h_T - \langle \Df \bX_T^{\bx}, \eta^{a,h} \rangle_{L^2(\Tt;\RR^m)}=\Delta^{\eta^{a,h}}_T=0$. 
    Similarly to \eqref{eq:IBP_computation} we get
    \begin{align*}
        D\Phi(\bx)(h) = \EE\left[\nabla\varphi(\bX^{\bx}_T) \bY^h_T\right]
        &= \EE\left[\nabla\varphi(\bX^{\bx}_T) \langle \Df \bX_T^{\bx}, \eta^{a,h} \rangle_{L^2(\Tt;\RR^m)} \right] \\
        &= \EE\left[ \langle \Df \varphi(\bX_T^{\bx}), \eta^{a,h} \rangle_{L^2(\Tt;\RR^m)} \right] \\
        &= \EE\left[ \varphi(\bX_T^{\bx}) \int_{\sft}^T \langle \eta^{a,h}_s, \D W_s\rangle\right].
    \end{align*}

    \emph{\underline{Step 2.}} The general case where $\varphi(\bX^{\bx}_T)\in L^p(\Omega)$ is obtained via a density argument, in a similar fashion as Step 3 of the proof of Theorem~\ref{th:main_IBP}. Let $(\varphi_n)_{n\in\NN}$ be a bounded approximating sequence of~$\varphi$, uniformly on compacts of~$\RR^d$. Then by a standard localisation argument, H\"older's and Markov's inequalities, we obtain for any~$R>0$
    \begin{align*}
        \EE[\abs{\varphi_m(\bX^{\bx}_T)-\varphi_n(\bX^{\bx}_T)}^2] 
        &\le \sup_{\abs{x}\le R} \abs{\varphi_m(x)-\varphi_n(x)}^2 + \norm{\varphi(\bX^{\bx}_T)}_{L^p(\Omega)}^{2/p}\PP(\abs{\bX^{\bx}_T}>R)^\frac{p-2}{p} \\
        &\lesssim \sup_{\abs{x}\le R} \abs{\varphi_m(x)-\varphi_n(x)}^2 +\left( \frac{1+\norm{\bx}_\infty}{R}\right)^\frac{p-2}{p}.
    \end{align*}
    It also holds, by Cauchy--Schwarz inequality,
    \begin{align}\label{eq:bound_etadW}
        \EE\left[\left(\int_{\sft}^T \langle\eta^{a,h}_s,\D W_s\rangle\right)^2\right] = \EE\int_{\sft}^T \abs{\eta^{a,h}_s}^2\ds 
        \le C \left(\int_{\sft}^T \abs{h(s)}^2 \ds+ \abs{h(T)}^2\right)\le C \norm{h}_{\infty}^2,
    \end{align}
    for $C>0$ independent of~$\bx$ and of $h$. 
    Therefore, for any $\ep>0$, there exist $R\in\NN$ and $n_1\in\NN$ large enough such that~$\abs{D\Phi_n(\bx)(h)-\EE\left[\varphi(\bX^{\bx}_T) \int_{\sft}^T \langle \eta^{a,h}_s,\D W_s\rangle\right]}\le \ep \norm{h}_\infty$ for all~$n\ge n_1$. Similarly to~\eqref{eq:est_PhinPhi_m} there is also~$n_2\in\NN$ such that~$\abs{(\Phi_n-\Phi_m)(\bx+h)-(\Phi_n-\Phi_m)(\bx)}\le \ep \norm{h}_\infty$. Finally, as in~\eqref{eq:frechet_density}, we conclude that~$D\Phi(\bx)(h)=\EE\left[\varphi(\bX^{\bx}_T) \int_{\sft}^T \langle \eta^{a,h}_s,\D W_s\rangle\right]$. This operator is bounded in virtue of~\eqref{eq:bound_etadW}.
    
    \emph{\underline{Step 3.}}
    Similarly to~\eqref{eq:bound_etadW}, two application of Cauchy--Schwarz inequality yield
    \begin{align*}
        \abs{\lim_{\delta\searrow0} D\Phi(\bx)(h^\delta) - \EE\left[\varphi(\bX^{\bx}_T) \int_{\sft}^T \langle \eta_s^{a,h},\D W_s\rangle\right] } \lesssim \norm{h-h^\delta}_{L^2(\Tt;\RR^d)}+\abs{h(T)^\delta-h(T)}.
    \end{align*}
    The proof is thus complete.
\end{proof}
Let $\Wh_t^{\bx} :=\bx(t)+ \int_{\sft}^t K_\sigma(t,s)\D W_s$ for all~$t\in\Tt$ (i.e. $\Wh^{\bx}=\bX^{\bx}$ with $b=0,\sigma=1$). Denoting~$\chi^2={\rm Var}(\Wh^{\bx}_T)=\int_{\sft}^T K_\sigma(T,s)^2\ds$, it is well-known by a standard integration by parts that for any~$h\in\Xx$ we have
\begin{align*}
    D\Phi(\bx)(h) = \frac{\EE\big[\varphi(\Wh^{\bx}_T)\Wh^{\bx}_T\big] h(T)}{\chi^2} = \EE\left[\varphi(\Wh^{\bx}_T) \int_{\sft}^T \frac{K_\sigma(T,s) h(T)}{\chi^2}\D W_s\right].
\end{align*}
We recovered such an IBP in Proposition \ref{prop:IBP_additive} by noticing that, in the state-dependent case, we only require the identity~$\bY^{h}_T = \langle \Df\bX^{\bx}_T,\eta\rangle_{L^2(\Tt)}$.


\subsection{Is there room for alternative formulae?}\label{subsec:optimality}
This section intends to explore whether other Malliavin weights exist and whether the condition $h\in\Hh$ is sharp. For simplicity we fix $d=1$ throughout but most arguments carry over to the multidimensional case. \smallskip

\noindent\textbf{The path-dependent case.}
The following (heuristic) argument consists in taking the proof backward to check if there is room for another sort of IBP formula. Let $(\Ff_t)_{t\in\Tt}$ be the filtration generated by $W$. Assume that, for~$\bx,h\in \Xx$ and $\phi\in\Cc^1_{b,\Xx}(\Xx)$, there exists an $\Ff_T$-measurable, centered, square integrable random variable~$\pi$ such that the integration by parts holds
    \begin{align}\label{eq:IBP_necessity}
        D\Phi(\bx)(h)
        = \EE\left[\phi(\bX^{\bx})\pi\right].
    \end{align}
By the martingale representation theorem \cite[Theorem IV-36.1]{Rogers00b}, there is an adapted process~$\eta\in L^2(\Omega\times\Tt)$ such that $\pi=\int_\sft^T  \eta_r \D W_r$. The same argument as in Lemma~\ref{lemma:equality_Y_DX} shows that the right hand side is equal to~$ \EE\left[D \phi(\bX^{\bx})\left(\bY^{\int_\sft^T K^r_\sigma \sigma(r,\bX^{\bx}_r) \eta_r\dr}\right)\right]$ whereas the left hand side, by the chain rule, is $\EE\left[D \phi(\bX^{\bx})\big(\bY^{h}\big)\right]$. This \textbf{does not necessarily imply} that $\bY^{h}=\bY^{\int_\sft^T K^r_\sigma \sigma(r,\bX^{\bx}_r) \eta_r\dr}$. Nevertheless we can conjecture that, if the IBP~\eqref{eq:IBP_necessity} holds for a large enough class of functions~$\phi$ and with the same weight~$\pi$, and if $\bX^{\bx},\bY^{h}$ and $W$ generate the same filtrations (which is the case if the noise is additive), then it must hold that~$\bY^{h-\int_\sft^T K^r_\sigma \sigma(r,\bX^{\bx}_r) \eta_r\dr}\equiv0$ almost surely.  
This would further entail that~$h=\int_\sft^T K^r_\sigma \sigma(r,\bX^{\bx}_r) \eta_r\dr$ and, assuming that~$K_\sigma^*$ is injective, that $h\in\Hh$ and $\eta_r=\xi(r,\bX^{\bx}_r)h^*(r)$.

\smallskip

The IBP formula \eqref{eq:main_IBP} is still valid after perturbing the stochastic integral $\pi_0=\int_\sft^T \big\langle\xi(r,\bX_r^{\bx})h^*(r), \D W_r\big\rangle$ with any random variable orthogonal to $\bX^{\bx}$. More precisely, $D\Phi(\bx)(h)=\EE[\phi(\bX^{\bx})\pi]$ holds for all~$\pi$ such that~$\EE[\pi\lvert\bX^{\bx}]=\EE[\pi_0\lvert\bX^{\bx}]$ (where $\EE[\cdot\lvert\bX^{\bx}]$ denotes conditional expectation with respect to the $\sigma$-algebra generated by~$\bX^{\bx}$). 
    This opens the gate to identifying the ``best" weight; see for instance~\cite{fournie2001applications} where it is shown that \textbf{the minimal variance weight} is $\EE[\pi_0\lvert\bX^{\bx}]$. In the case of additive noise~$\sigma\equiv\sigma_0$ and power-law kernel $K_\sigma=K_H$, we know that~$W$ generates the same filtration as~$\bX^{\bx}$. Therefore $\pi_0=\EE[\pi_0\lvert\bX^{\bx}]$ is in fact the optimal weight.

\medskip

\noindent{\textbf{The state-dependent case.}}
So far we restrained ourselves to adapted processes~$\eta$ but the theory does go further.
Indeed, the Malliavin integration by parts is not limited to the Itô integral. The divergence operator~$\bm{\delta}$ is the adjoint of the derivative~$\Df$. More precisely, the so-called Skorohod integral $\bm{\delta}(u)$ is the element of~$L^2(\Omega)$ characterised by~$\EE[F\bm{\delta}(u)]=\EE[\int_{\sft}^T \Df_sF u_s\ds]$, for any~$F\in\DD^{1,2}$ and~$u\in L^2(\Omega\times\Tt)$ wherever this is well-defined. We refer to \cite[Section 1.3]{nualart2006malliavin} for more details. 
For any $X_T,Y_T\in L^1(\Omega)$ and $\eta:\Tt\to L^2(\Omega)$ such that everything below is well-defined we have
\begin{align*}
    \EE[\varphi'(X_T)Y_T] 
    &= \EE\left[\int_{\sft}^T \varphi'(X_T)\frac{\eta_r \Df_r X_T Y_T}{\int_{\sft}^T \eta_s \Df_s X_T\ds} \dr\right] 
    = \EE\left[\varphi(X_T) \bm{\delta}\left(\frac{\eta_\cdot Y_T}{\int_{\sft}^T \eta_s \Df_s X_T\ds}\right)\right] \\
    &= \EE\left[\varphi(X_T) \frac{ Y_T}{\int_{\sft}^T \eta_s \Df_s X_T\ds}\bm{\delta}\left(\eta_\cdot\right)\right]
    + \EE\left[\varphi(X_T) \int_{\sft}^T \eta_r \Df_r \left\{\frac{ Y_T}{\int_{\sft}^T \eta_s \Df_s X_T\ds}\right\}\dr\right],
\end{align*}
where the last equality can be found in~\cite[Proposition 1.3.3]{nualart2006malliavin}. Setting~$X = \bX^{\bx}$ and~$Y=\bY^{h}$, for some~$h\in\Xx$, then the left side is nothing else than~$D\Phi(\bx)(h)$. On the right side, the second term vanishes if and only if~$\bY^{h}_T = \langle \Df \bX^{\bx}_T, \eta\rangle_{L^2(\Tt)}$.


\section{A family of fractional integration by parts} \label{sec:IBP_frac}

We would like to extend the class of directions~$\Hh$ for which we could show differentiability of~$\Phi$ in the previous section. Indeed, by Cauchy--Schwarz inequality it follows that any~$h\in\Hh$ must satisfy $h({\sft})=0$ and in particular constant directions (other than zero) are not included.  

This relative rigidity leads us to consider whether we could trade some regularity from~$\phi(\bX^{\bx})$ to the direction~$h$. It turns out that the adequate notion that gives rise to an interpolation between~$\phi$ and~$D\phi$ is the right Riemann--Liouville derivative, denoted~$\scrD^\alpha_{T^-}$ and defined in~\eqref{eq:def_fracD_right}. 
In analogy with~$\Hh$, we introduce a new class $\Hh_\alpha$ of admissible directions
\begin{align*}
    &\Hh_\alpha := \left\{h:\Tt\to\RR^d \,\lvert\, h(t)=\int_{\sft}^t K_\sigma(t,s) h^*(s) \ds, \, h^*\in L^2_\alpha\right\}\cap\, \Xx, \\
    & L^2_\alpha := \left\{h:\Tt\to\RR^d \,\lvert\, \int_{\sft}^T (s-\sft)^{2\alpha} \abs{h^*(s)}^2\ds <\infty\right\}.
\end{align*}
Let us define the associated norms~$\norm{h}_{\Hh_\alpha}^2:=\norm{h^*}_{L^2_\alpha}^2:=\int_{\sft}^T (s-\sft)^{2\alpha} \abs{h^*(s)}^2\ds $. 
Further define the martingale~$M$ as~$M_t:=\EE[\phi(\bX^{\bx})\lvert\Ff_t]$ for all~$t\in\Tt$ and in particular~$M_T=\phi(\bX^{\bx})$. We denote its quadratic variation by~$\langle M\rangle$.
\begin{theorem}\label{thm:frac_IBP}
    Let Assumption \ref{as:coefs} and \ref{as:kernel} hold. 
    \begin{enumerate}
        \item[A)] Let $\phi\in\Cc_{b,\Xx,{\rm Lip}}^1(\Xx)$ and suppose that there exists~$\alpha\in(0,1)$ such that 
    \begin{align}\label{eq:fractional_condition}
    \EE\left[\int_{\sft}^T (s-\sft)^{-2\alpha}\D \langle M\rangle_s\right] 
    <\infty.
    \end{align}
    Then $\scrD^{\alpha}_{T^-}\big(M_T-M_\cdot)(\sft)=\int_{\sft}^T (s-\sft)^{-\alpha}\D M_s$ belongs to~$L^2(\Omega)$, we have~$\Phi\in\Dd^1_{\Hh_\alpha}(\Xx)$ and for any $h\in\Hh_\alpha$ the fractional integration by parts holds
    \begin{align}\label{eq:frac_IBP_formula}
    D\Phi(\bx)(h) 
    &= \Gamma(1-\alpha) \EE\left[\scrD^{\alpha}_{T^-}\Big(M_T-M_\cdot\Big)(\sft) \int_\sft^T (s-\sft)^\alpha \Big\langle \xi(s,\bX_s^{\bx}) h^*(s),  \D W_s\Big\rangle \right].
\end{align}
    \item[B)] Suppose that~$\phi(\bx)=\varphi(\bx(T))$ for some $\varphi\in \Cc^\beta_{{\rm poly}}(\RR^d)$, $\beta\in(0,1)$,  and let $\alpha<\beta/2$.  
    Then $\scrD^{\alpha}_{T^-}\big(M)(\sft)$ exists and belongs to~$L^2(\Omega)$, we have~$\Phi\in\Dd^1_{\Hh_\alpha}(\Xx)$ and for any $h\in\Hh_\alpha$ the fractional integration by parts~\eqref{eq:frac_IBP_formula} holds. Moreover, we have the representation
\begin{align*}
    \scrD^\alpha_{T^-}(M_T-M_\cdot)(\sft) = \frac{1}{\Gamma(1-\alpha)}\left( \frac{M_T-M_{\sft}}{(T-\sft)^\alpha} + \alpha \int_{\sft}^T \frac{M_t-M_{\sft}}{(t-\sft)^{1+\alpha}}\dt\right).
\end{align*}
    \item[C)] Under the same assumptions as B), for any $\bx\in\Xx, h\in\Hh_\alpha$ we have
    \begin{align}
        \abs{\Phi(\bx+h)-\Phi(\bx)} \lesssim \norm{h}_{\Hh_\alpha}.
    \end{align}
    \end{enumerate}
\end{theorem}
The proofs of these results are gathered in Sections~\ref{sec:proof_fIBP_A}, \ref{sec:proof_fIBP_B} and \ref{sec:proof_fIBP_C} respectively. Beforehand, let us make a few observations about the class of admissible direction~$\Hh_\alpha$ and the existence of the fractional derivative of $M$. Note that the lower $\alpha$ is, the less regular the test function~$\varphi$ must be and the more regular the direction~$h$ must be.
\begin{remark}[On the class $\Hh_\alpha$]\label{rem:Halpha}
In this remark we will set ~$K_\sigma=K_H$ to illustrate our results. Recall that in this case $h^* = \mathscr{D}^{H+1/2}_{\sft^+}(h) / \Gamma(H+1/2)$. We discuss which directions are admissible and for which $\alpha$.
\begin{enumerate}
    \item[i)] Consider a constant direction~$h\equiv1$. Then~$h^*(t)=c(t-\sft)^{-H-\half}$, where $c=(\Gamma(H+1/2)\Gamma(1/2-H))^{-1}$, and hence~$h\in\Hh_\alpha$ for all~$\alpha>H$. The map~$\Phi$ is thus differentiable in constant directions provided that $\varphi\in\Cc^\beta_{{\rm poly}}(\RR^d)$ with~$\beta>2H$.
    \item[ii)] More generally, let us consider $h(t)=(t-\sft)^{\gamma-1/2} g(t)$ where $g\in\Cc^{\beta/2}$ and $1/2-\gamma+H <\alpha < \beta/2$. Then by \cite[Theorem 13.6]{samko1993fractional} we have $\abs{h^*(t)}\lesssim\lvert\mathscr{D}^{H+1/2}_{\sft^+} (h)(t)\lvert\lesssim (t-\sft)^{\gamma-H-1}$ and hence 
    $$
\norm{h}_{\Hh_\alpha}^2 = \norm{h^*}_{L^2_\alpha}^2 \lesssim \int_{\sft}^T (s-\sft)^{2(\alpha+\gamma-1/2-H)-1}\ds<+\infty.
    $$
    If moreover $h$ is continuous then it belongs to $\Hh_\alpha$ and is an admissible direction.
    \item[iii)] This last point raises an important question: what about unbounded directions? As long as~$\norm{h}_{\Hh_\alpha}<+\infty$ the right-hand-side of~\eqref{eq:frac_IBP_formula} is well-posed and finite. Yet the domain of $\Phi$ is $\Xx$ hence differentiating in a non-continuous direction leads to ambiguities. The natural notion, that arises both in~\cite{viens2019martingale} and~\cite{gasteratos2025kolmogorov}, consists in defining the derivative in an unbounded direction~$h$ as the limit of derivatives in continuous directions~$h^\delta$ with~$h^\delta\to h$ as $\delta\to0$.
    From that point of view, we show in Corollary \ref{coro:fIBP_singular} that $h(t)=(t-\sft)^{\gamma-1/2}$ is indeed an admissible direction for~$\gamma>H+1/2-\alpha$.
    \item[iv)] Differentiating in the direction of the kernel $h(t)=K_H(t,\sft)=(t-\sft)^{H-1/2}$, as required by the functional Itô formula, remains outside the scope of our method. If we wanted to define this derivative as in the previous item we would first define $h^\delta(t)=h(t\vee(\sft+\delta))$ but then notice that, for~$\norm{h^\delta}_{\Hh_\alpha}$ to remain bounded as $\delta$ goes to zero, we would require~$\alpha>1/2$. More importantly, solving for $K_H(t,\sft)=\int_{\sft}^t K_H(t,s) h^*(s)\ds$ would informally force $h^*$ to be a Dirac mass, rendering the stochastic integral meaningless.  
\end{enumerate}
\end{remark}

\begin{remark}[On the fractional derivative]
For $\varphi\in\Cc^1_b(\RR^d)$ the Clark--Ocone formula reads
\begin{align*}
M_T-M_t = \int_t^T \EE[\Df_s\varphi(\bX^{\bx}_T)\lvert\Ff_s]\D W_s    
\end{align*}
and in virtue of the estimate $\EE\abs{\Df_s\varphi(\bX_T^{\bx})}^2\lesssim \EE\abs{\Df_s\bX_T^{\bx}}^2\lesssim K_H(T,s)^2$ we obtain
\begin{align}\label{eq:DalphaM_condition_CO}
    \EE\left[\int_{\sft}^T (s-\sft)^{-2\alpha} \D \langle M\rangle_s \right] \lesssim \int_{\sft}^T (s-\sft)^{-2\alpha} (T-s)^{2H-1}\ds \lesssim (T-\sft)^{2(H-\alpha)}.
\end{align}
for all $\alpha\in(0,1/2)$. This immediately yields a sufficient condition for \eqref{eq:frac_IBP_formula} to hold.

If $\varphi$ is not differentiable, another condition for $\scrD^\alpha_{T^-}(M-M_T)(\sft)$ to exist would be that~$M\in\Cc^{\alpha+\ep}$ almost surely, for some~$\ep>0$, see \cite[Lemma 13.1]{samko1993fractional} for the left fractional derivative. However, $M$ is only expected to be in~$\Cc^{H-\ep}$ which would restrict to~$\alpha<H$ and rule out constant directions (see Remark~\ref{rem:Halpha} i)). Equation~\eqref{eq:DalphaM_condition_CO} however suggests to dive deeper into the study of the regularity. As the proof of B) demonstrates we care about the regularity of $M$ around $\sft$, which we prove in~\eqref{eq:est_Mt_reg} to be of the order $\beta/2$ when $\varphi\in\Cc^\beta_{{\rm poly}}(\RR^d)$. We lose a factor $2$ because of the Brownian motion but the singularity of the kernel only intervenes around $T$, as in~\eqref{eq:DalphaM_condition_CO}. We then obtain, provided $\beta>2\alpha$,
\begin{align*}
    \EE\left[\int_{\sft}^T (s-\sft)^{-2\alpha} \D \langle M\rangle_s \right] \lesssim \int_{\sft}^T (s-\sft)^{\beta-2\alpha-1} (T-s)^{\beta(2H-1)}\ds \lesssim (T-\sft)^{2(\beta H-\alpha)}.
\end{align*}
Further notice that, if $h(t)=(t-\sft)^{\gamma-1/2}$ with $\gamma>H+1/2-\alpha$, then $\norm{h}_{\Hh_\alpha}\lesssim (T-\sft)^{\gamma-1/2 + \alpha-H}$. Combining these estimates, we obtain under the assumptions of Theorem~\ref{thm:frac_IBP} B) or Corollary~\ref{coro:fIBP_singular} (depending whether or not $\gamma\ge1/2$) the gradient bound
    \begin{align}\label{eq:gradient_bound}
        \abs{D\Phi(\bx)(h)} \le \EE\left[\int_{\sft}^T (s-\sft)^{-2\alpha}\D \langle M\rangle_s\right]^\half \norm{\xi}_\infty \norm{h}_{\Hh_\alpha}
        \lesssim (T-\sft)^{\gamma-1/2+H(\beta-1)}.
    \end{align}
    When $\varphi\in\Cc^1_{b}(\RR^d)$ we easily obtain by the chain rule a bound of the form~$C(T-\sft)^{\gamma-1/2}$, hence~\eqref{eq:gradient_bound} highlights the regularity loss induced by the lack of differentiability of the test function.
\end{remark}
This follow-up result deals with singular derivatives in the case of unbounded power-law kernel and direction.
\begin{corollary}\label{coro:fIBP_singular}
    Let Assumption \ref{as:coefs} hold. Further set $0<H<\alpha<\beta/2<1/2$,  $K_\sigma=K_H$ and~$\phi(\bx)=\varphi(\bx(T))$ for some $\varphi\in \Cc^\beta_{{\rm poly}}(\RR^d)$. Let $h(t)=(t-\sft)^{\gamma-1/2}$ and $h^\delta(t) = (\delta\vee (t-\sft))^{\gamma-1/2}$ for all~$\delta>0$, where $\gamma>1/2+H-\alpha$. Then $h^\delta\in\Hh_\alpha$ and the following limit holds:
    \begin{align*}
        \lim_{\delta\searrow0} D\Phi(\bx)(h^\delta)
        = \Gamma(1-\alpha) \EE\left[\scrD^{\alpha}_{T^-}\Big(M_T-M_\cdot\Big)(\sft) \int_\sft^T (s-\sft)^\alpha \Big\langle \xi(s,\bX_s^{\bx}) h^*(s),  \D W_s\Big\rangle \right].
    \end{align*}
    \end{corollary}
\begin{proof}
\emph{\underline{Step 1: $\norm{h}_{\Hh_\alpha}<\infty$.}} Even though $h$ is not continuous, we start by proving that its $\Hh_\alpha$ norm is finite. Since $K_\sigma=K_H$ we have 
\begin{align*}
    h^{*}(s)=\scrD^{H+1/2}_{\sft^+}(h)(s)
    = \frac{1}{\Gamma(1/2-H)}\frac{\D}{\ds}\int_{\sft}^s (s-r)^{-H-1/2}(r-\sft)^{\gamma-1/2} \dr =c (s-\sft)^{\gamma-H-1},
\end{align*}
for some constant $c>0$. Therefore
\begin{align}\label{eq:powerlaw_est_h}
    \norm{h}_{\Hh_\alpha}^2 = c^2 \int_{\sft}^T (s-\sft)^{2(\alpha+\gamma-H-1)}\ds <\infty
\end{align}
because, by assumption, $\gamma>1/2+H-\alpha$.

\emph{\underline{Step 2: $h^\delta\in\Hh_\alpha$.}} 
Unlike $h$, it is clear that $h^\delta$ is continuous. Since $K_\sigma=K_H$ we have 
\begin{align*}
    h^{\delta,*}(s)=\scrD^{H+1/2}_{\sft^+}(h^\delta)(s)
    = \frac{1}{\Gamma(1/2-H)}\frac{\D}{\ds}\int_{\sft}^s (s-r)^{-H-1/2}h^\delta(r) \dr.
\end{align*}
In the case $s\in[\sft,\sft+\delta]$ we have $h^{\delta,*}(s)= \frac{\delta^{\gamma-1/2} (s-\sft)^{-H-\half}}{\Gamma(1/2-H)}$ and hence
\begin{align}\label{eq:hdeltastar_1}
    \int_{\sft}^{\sft+\delta} \abs{h^{\delta,*}(s)}^2 (s-\sft)^{2\alpha}\ds 
    \lesssim \delta^{2(\alpha-H+\gamma-1/2)},
\end{align}
since $\gamma>1/2+H-\alpha$ by assumption.

Moreover, $h=h^\delta$ on~$[\sft+\delta,T]$ and, for $r\in[\sft,\sft+\delta]$ we have, for any~$\mu>0$,
\begin{align*}
    (r-\sft)^{\gamma-1/2}-\delta^{\gamma-1/2} = (1/2-\gamma)\int_r^{\sft+\delta} (u-\sft)^{\gamma-3/2}\du 
    &\lesssim \int_r^{\sft+\delta} (u-r)^{\mu-1}\du (r-\sft)^{\gamma-1/2-\mu} \\
    &\lesssim (\sft+\delta-r)^\mu (r-\sft)^{H-1/2-\mu}\delta^{\gamma-H},
\end{align*}
where we also used $\gamma-H>0$. 
Hence for $s\in(\sft+\delta,T]$ we obtain
\begin{align*}
    \abs{h^*(s)-h^{\delta,*}(s) }
    &= \abs{\frac{\D}{\ds}\int_{\sft}^{\sft+\delta} (s-r)^{-H-1/2}((r-\sft)^{\gamma-1/2}-\delta^{\gamma-1/2}) \dr}\\
    &\lesssim \delta^{\gamma-H} \int_{\sft}^{\sft+\delta} (s-r)^{-H-3/2}(\sft+\delta-r)^\mu (r-\sft)^{H-1/2-\mu}\dr \\
    &=\delta^{\gamma-H} \delta^{H+1/2} (s-\sft-\delta)^{\mu-H-1/2}(s-\sft)^{-\mu-1} B(1/2+H-\mu,\mu+1),
\end{align*}
where we used identity 3.199 in \cite{gradshteyn2007table} to compute the integral. Let $\mu\in (H,\gamma+\alpha-1/2)$, which is non-empty by assumption. Over this interval it entails
\begin{equation}
\begin{aligned}\label{eq:bound_h-hdelta}
    \int_{\sft+\delta}^T (h^*(s)-h^{\delta,*}(s))^2 (s-\sft)^{2\alpha}\ds 
    &\lesssim \delta^{2\gamma+1}  \int_{\sft+\delta}^T (s-\sft-\delta)^{2\mu-2H-1} (s-\sft)^{2\alpha-2\mu-2}\ds \\
    &\lesssim \delta^{2(\gamma+\alpha-\mu-1/2)} (T-\sft-\delta)^{2(\mu-H)}
\end{aligned}
\end{equation}
which is finite and tends to zero as $\delta$ goes to zero. Combining \eqref{eq:powerlaw_est_h}, \eqref{eq:hdeltastar_1} and \eqref{eq:bound_h-hdelta} entails
\begin{align*}
    \norm{h^\delta}_{\Hh_\alpha} 
    \le \norm{h^\delta\one_{[\sft,\sft+\delta]}} + \norm{(h^\delta-h)\one_{[\sft+\delta,T]}}_{\Hh_\alpha} + \norm{h}_{\Hh_\alpha}
    <\infty.
\end{align*}

\emph{\underline{Step 3: Convergence.}}
    By Cauchy--Schwarz inequality and $\scrD^{\alpha}_{T^-}\Big(M_T-M_\cdot\Big)(\sft)\in L^2(\Omega)$ we obtain
    \begin{align*}
        \abs{D\Phi(\bx)(h^\delta)
        - \Gamma(1-\alpha) \EE\left[\scrD^{\alpha}_{T^-}\Big(M_T-M_\cdot\Big)(\sft) \int_\sft^T (s-\sft)^\alpha \Big\langle \xi(s,\bX_s^{\bx}) h^*(s),  \D W_s\Big\rangle \right]}
        \lesssim \norm{h-h^\delta}_{\Hh_\alpha}.
    \end{align*}
Since $K_\sigma=K_H$ we have 
\begin{align*}
    h^*(s)-h^{\delta,*}(s)=\scrD^{H+1/2}_{\sft^+}(h-h^\delta)(s)
    = \frac{1}{\Gamma(1/2-H)}\frac{\D}{\ds}\int_{\sft}^s (s-r)^{-H-1/2}(h-h^\delta)_r \dr.
\end{align*}
In the case $s\in[\sft,\sft+\delta]$, integrating and then differentiating  yields
\begin{align*}
    \abs{h^*(s)-h^{\delta,*}(s) }
    &= \frac{1}{\Gamma(1/2-H)}\abs{\frac{\D}{\ds}\int_{\sft}^s (s-r)^{-H-1/2}((r-\sft)^{\gamma-1/2}-\delta^{\gamma-1/2}) \dr}\\
    &= \frac{1}{\Gamma(1/2-H)}\abs{\frac{\D}{\ds}\left( (s-\sft)^{\gamma-H} B(1/2-H,1/2+\gamma) - \frac{\delta^{\gamma-1/2}(s-\sft)^{1/2-H}}{1/2-H} \right) }\\
    &\lesssim (s-\sft)^{\gamma-H-1} + \delta^{\gamma-1/2}(s-\sft)^{-1/2-H}.
\end{align*}
On this interval we thus have
\begin{align*}
       \int_{\sft}^{\sft+\delta} (h^*(s)-h^{\delta,*}(s))^2 (s-\sft)^{2\alpha}\ds 
       &\lesssim \int_{\sft}^{\sft+\delta} (s-\sft)^{2(\gamma-H-1)+2\alpha} \ds + \delta^{2\gamma-1} \int_{\sft}^{\sft+\delta} (s-\sft)^{2\alpha-1-2H} \ds \\
       &\lesssim \delta^{2(\gamma-H+\alpha)-1},
\end{align*}
which tends to zero because $\gamma>1/2+H-\alpha$ by hypothesis. The interval $[\sft+\delta,T]$ was already dealt with in Step 2 \eqref{eq:bound_h-hdelta} hence the proof is complete.
\end{proof}

\subsection{Proof of Theorem \ref{thm:frac_IBP} A)}\label{sec:proof_fIBP_A}
Since~$\phi\in\Cc_{b,\Xx,{\rm Lip}}^1(\Xx)$, $\phi(\bX^{\bx})$ is Malliavin differentiable and the Clark--Ocone formula entails~ $M_t:=\int_\sft^t \big\langle\EE[\Df_r \phi(\bX^{\bx})\lvert \Ff_r] ,\D W_r\big\rangle$. Standard estimates and Kolmogorov continuity theorem then guarantee that~$M$ has almost surely continuous paths. Introduce the process~$\mathscr{M}^\alpha$ defined for all~$t\in[\sft,T]$ as~$\mathscr{M}^\alpha_t:=\int_t^T (r-t)^{-\alpha}\D M_r$. By assumption we have 
\begin{align*}
    \EE\abs{\mathscr{M}^\alpha_{\sft}}^2 
        = \EE\int_{\sft}^T (r-\sft)^{-2\alpha} \D \langle M\rangle_r <\infty
\end{align*}
and thus~$\mathscr{M}^\alpha_{\sft}\in L^2(\Omega)$. 
Resuming the proof of Theorem~\ref{th:main_IBP}  from~\eqref{eq:IBP_computation}, we are thus allowed to introduce a fractional kernel before applying the integration by parts, as follows:
\begin{align*}
    D \Phi(\bx)(h) 
    &= \EE\left[\int_\sft^T \EE[\Df_r \phi(\bX^{\bx})^\top\lvert \Ff_r] \xi(r,\bX^{\bx}_r) h^*(r)  \dr \right]\\
    &= \EE\left[\int_\sft^T  (r-\sft)^{-\alpha}\Big\langle \EE[\Df_r \phi(\bX^{\bx})\lvert \Ff_r],\D W_r \Big\rangle \int_\sft^T (r-\sft)^\alpha \Big\langle\xi(r,\bX^{\bx}_r) h^*(r) , \D W_r\Big\rangle \right]\\
    &= \EE\left[\mathscr{M}^\alpha_{\sft}  \int_\sft^T (r-\sft)^\alpha \Big\langle\xi(r,\bX^{\bx}_r) h^*(r) ,  \D W_r \Big\rangle\right].
\end{align*}
It remains to link $\mathscr{M}^\alpha$ with the fractional derivative defined in~\eqref{eq:def_fracD_right}. In virtue of the stochastic Fubini theorem we have
\begin{align}\label{eq:Malpha_eq_DalphaM}
    \scrI^\alpha_{T^-}(\mathscr{M}^\alpha)(t) 
    &= \frac{1}{\Gamma(\alpha)}\int_t^T (s-t)^{\alpha-1} \left(\int_s^T (r-s)^{-\alpha}\D M_r \right)\ds\\
    &=\frac{1}{\Gamma(\alpha)} \int_t^T \left( \int_t^r(s-t)^{\alpha-1} (r-s)^{-\alpha} \ds\right) \D M_r\nonumber\\
    &=\Gamma(1-\alpha) (M_T-M_t). \nonumber
\end{align}
Applying the fractional derivative to both sides entails~$\mathscr{M}^\alpha = \Gamma(1-\alpha) \scrD^{\alpha}_{T^-}(M_T-M_\cdot)$. This yields, as desired,
\begin{align*}
    D\Phi(\bx)(h) 
    &= \Gamma(1-\alpha) \EE\left[\scrD^{\alpha}_{T^-}\Big(M_T-M_\cdot\Big)(\sft) \int_\sft^T (r-\sft)^\alpha \Big\langle\xi(r,\bX^{\bx}_r) h^*(r),  \D W_r \Big\rangle\right].
\end{align*}

\subsection{Proof of Theorem \ref{thm:frac_IBP} B)}\label{sec:proof_fIBP_B}

There are three steps to this proof. In the first one we study the regularity properties of processes related to~$\bX^{\bx}$. These allow us, in the second step, to derive the necessary path regularity of the martingale~$M_t=\EE[\varphi(\bX^{\bx}_T)\lvert\Ff_t]$, thereby securing condition~\eqref{eq:fractional_condition}. Under the assumption that~$\varphi\in\Cc^1_b(\RR^d)$, this guarantees the fractional integration by parts holds by Theorem~\ref{thm:frac_IBP} A). The last step consists in relaxing the latter assumption by an approximation argument. 

\emph{\underline{Step 1: Regularity estimates.} }
The analysis of $M$ requires the introduction of new stochastic processes and some preliminary results pertaining to them. Let us introduce the family of SVEs indexed by~$t\in\Tt,\bx\in \Cc([t,T];\RR^d)$,
    \begin{align*}
        \bX^{t,\bx}_\tau = \bx(\tau) + \int_{t}^\tau K_b(\tau,s)b(s,\bX_s^{t,\bx})\ds + \int_t^\tau K_\sigma(\tau,s)\sigma(s,\bX_s^{t,\bx})\D W_s, \quad \tau\in[t,T].
    \end{align*}
    For each $t\in\Tt$ we also define the auxiliary~$\Ff_t$-measurable process~$\widetilde{\bX}^{t,\bx}$ indexed by 
    \begin{align*}
        \widetilde{\bX}^{t,\bx}_\tau := \bx(\tau) + \int_{\sft}^t K_b(\tau,s)b(s,\bX_s^{\sft,\bx})\ds + \int_\sft^t K_\sigma(\tau,s)\sigma(s,\bX_s^{\sft,\bx})\D W_s,\quad \tau\in[t,T]. 
    \end{align*}
    Pathwise uniqueness yields~$\bX^{\sft,\bx}=\bX^{\bx}$ as defined by Equation~\eqref{eq:main_SVE} and it also entails the flow property~$\bX^{\sft,\bx}=\bX^{t,\widetilde{\bX}^{t,\bx}}$ for all~$t\in\Tt$ almost surely since the SVE satisfied by~$\bX^{\bx}$ can be reformulated as
    \begin{align*}
        \bX^{\bx}_\tau = \widetilde{\bX}^{t,\bx}_\tau + \int_t^\tau K_b(\tau,s)b(s,\bX_s^{\sft,\bx})\ds + \int_t^\tau K_\sigma(\tau,s)\sigma(s,\bX_s^{\sft,\bx})\D W_s, \quad \tau\in[t,T].
    \end{align*} 
\begin{lemma}\label{lemma:est_X}
    Under Assumptions~\ref{as:kernel} and~\ref{as:coefs}, we have for all $\sft\le s\le t< \tau\le T$, $\bx,\bm{y}\in\Xx$ and $p\ge2$
    \begin{align}
        &\EE\abs{\bX^{t,\bx}_\tau-\bX^{t,\bm{y}}_\tau}^p \lesssim \abs{\bx(\tau)-\bm{y}(\tau)}^p+ \int_t^\tau (\tau-r)^{2H-1} \abs{\bx(r)-\bm{y}(r)}^p\dr; \label{eq:est_reg_XxXy}\\
        &\EE\abs{\bX^{t,\bx}_\tau-\bX^{s,\bx}_\tau}^2
        \lesssim (t-s) (\tau-t)^{2H-1} \label{eq:est_reg_XtXsft}\\
        &\EE\abs{\widetilde{\bX}^{t,\bx}_\tau - \widetilde{\bX}^{s,\bx}_\tau}^p \lesssim (t-s)^{p/2} (\tau-t)^{(H-1/2)p}. \label{eq:est_reg_Xtildets}
    \end{align}
\end{lemma}
\begin{proof}
    \textbf{Proof of~\eqref{eq:est_reg_XxXy}.} The proof is identical to Step 1 of the proof of Lemma~\ref{lemma:cty_Z}, in particular Equation~\eqref{eq:bound_DeltaX} with the estimate~$\abs{R_{C,H}(t,s)}\lesssim K_H(\tau,s)^2= (t-s)^{2H-1}$.

    \textbf{Proof of~\eqref{eq:est_reg_XtXsft}.} We observe that for all~$\sft\le s<t<\tau$ 
    \begin{align*}
        \EE\abs{\bX^{t,\bx}_\tau-\bX^{s,\bx}_\tau}^2 
        &\lesssim \EE\abs{\int_t^\tau K_b(\tau,r)\Big(b(\bX^{t,\bx}_r)-b(\bX^{s,\bx}_r)\Big)\dr}^2 
        + \EE\abs{\int_s^t K_b(\tau,r)b(\bX^{s,\bx}_r)\dr}^2 \\
        & + \EE\abs{\int_t^\tau K_\sigma(\tau,r)\Big(\sigma(\bX^{t,\bx}_r)-\sigma(\bX^{s,\bx}_r)\Big)\D W_r}^2 
        + \EE\abs{\int_s^t K_\sigma(\tau,r)\sigma(\bX^{s,\bx}_r)\D W_r}^2 \\
        &\lesssim \EE\int_t^\tau K_H(\tau,r)^2 \abs{\bX^{t,\bx}_s-\bX^{s,\bx}_r}^2 \dr+ \int_s^t K_H(\tau,r)^2\dr\\
        &\lesssim \int_t^\tau K_H(\tau,r)^2 \EE\abs{\bX^{t,\bx}_r-\bX^{s,\bx}_r}^2 \dr +(t-s)(\tau-t)^{2H-1}
    \end{align*}
    where we used Jensen's and BDG inequalities, the boundedness and Lipschitz continuity of~$b,\sigma$ and Assumption~\ref{as:kernel}. 
    We apply the Volterra--Gr\"onwall inequality from Lemma~\ref{lemma:Grownall} where we denote by $R_{C,H}$ the resolvent of~$CK_H^2$ \eqref{eq:Res_powerlaw} and use that~$R_{C,H}(\tau,r)\lesssim (\tau-r)^{2H-1}$:
    \begin{align*}
        \EE\abs{\bX^{t,\bx}_\tau-\bX^{s,\bx}_\tau}^2 
        &\lesssim (t-s)(\tau-t)^{2H-1} + \int_t^\tau R_{C,H}(\tau,r)(t-s) (r-t)^{2H-1}\dr \\
        &\lesssim (t-s) \left((\tau-t)^{2H-1} + \int_t^\tau (\tau-r)^{2H-1} (r-t)^{2H-1}\dr \right)\\ 
        & \lesssim (t-s) (\tau-t)^{2H-1}.
    \end{align*}
    The last line follows from the identity $\int_t^\tau (\tau-r)^{2H-1} (r-t)^{2H-1}\dr = (\tau-t)^{4H-1} {\rm B}(2H,2H)$. 
    
    \textbf{Proof of~\eqref{eq:est_reg_Xtildets}.} For $\sft\le s\le t<\tau$, similar arguments yield
    \begin{align*}
        \EE\big\lvert \widetilde{\bX}^{t,\bx}_\tau-\widetilde{\bX}^{s,\bx}_\tau\big\lvert ^{p} 
        &\le \EE\abs{\int_s^t K_b(\tau,r)b(r,\bX^{\bx}_r)\ds + \int_s^t K_\sigma(\tau,r)\sigma(r,\bX^{\bx}_r)\D W_r}^p \\
        &\lesssim \left(\int_{s}^t K_H(\tau,r)^2\ds\right)^{p/2} \le (t-s)^{p/2} (\tau-t)^{(2H-1)p/2},
    \end{align*}
    thereby concluding the proof.
\end{proof}

\underline{\emph{Step 2: The fractional derivative.}} In this step we will prove the following lemma which may be of independent interest. Recall that~$M_t=\EE[\varphi(\bX_T^{\bx})\lvert\Ff_t]$ for some~$\varphi\in\Cc^\beta_{{\rm poly}}(\RR^d)$ and $\mathscr{M}^\alpha_t = \int_t^T (s-t)^{-\alpha}\D M_s$.
\begin{lemma}\label{lemma:DalphaM_is_Malpha}
    Let Assumptions~\ref{as:kernel}  and~\ref{as:coefs} hold. Then $\mathscr{D}^\alpha_{T^-}(M_T-M_\cdot)(\sft) = \mathscr{M}^\alpha_{\sft}/\Gamma(1-\alpha)$ and belongs to $L^2(\Omega)$ for all~$\alpha<\beta/2$.
\end{lemma}
\begin{proof}
Let us define the family of measurable maps~$m_t:\Cc([t,T];\RR^d)\to\RR$ as~$m_t(\bx):=\EE[\varphi(\bX^{t,\bx}_T)]$, $t\in\Tt$, and we claim that~$M_t=m_t(\widetilde{\bX}^{t,\bx})$. Indeed we have
\begin{align*}
    m_t(\widetilde{\bX}^{t,\bx}) = \EE\big[\varphi\big(\bX_T^{t,\zeta}\big)\big]\lvert_{\zeta=\widetilde{\bX}^{t,\bx}}
= \EE\Big[\varphi\Big(\bX_T^{t,\widetilde{\bX}^{t,\bx}}\Big) \Big\lvert \widetilde{\bX}^{t,\bx}\Big]
= \EE\Big[\varphi\Big(\bX_T^{t,\widetilde{\bX}^{t,\bx}}\Big) \Big\lvert \Ff_t\Big]
=\EE\big[\varphi\big(\bX^{\bx}_T\big)\big\lvert \Ff_t\big],
\end{align*}
where we used the flow property~$\bX^{\bx}=\bX^{t,\widetilde{\bX}^{t,\bx}}$ (almost surely).

Defining the (non-negative) measure $\mu((s,t]):=\EE[\langle M\rangle_t - \langle M\rangle_s]$ we have by the layer-cake representation

\begin{align}\label{eq:rpz_Malpha}
        \EE\abs{\mathscr{M}^\alpha_{\sft}}^2 
        &= \EE\int_{\sft}^T (s-\sft)^{-2\alpha} \D \langle M\rangle_s 
        = \int_\sft^T (s-\sft)^{-2\alpha}\mu(\D s)\\ 
        &= \int_0^\infty \mu\Big(\Big\{ s\in(\sft,T] \,\lvert\, (s-\sft)^{-2\alpha}>x\Big\}\Big) \D x 
        =\int_0^\infty \mu\big((\sft,(\sft+x^{-\frac{1}{2\alpha}})\wedge T) \big)\D x \nonumber\\
        &=2\alpha\int_{\sft}^\infty \mu\big((\sft,t\wedge T]\big) (t-\sft)^{-2\alpha-1}\D t 
        \nonumber\\
        &= 2\alpha \int_{\sft}^T \EE\Big[\langle M\rangle_{t}-\langle M\rangle_{\sft}\Big] (t-\sft)^{-2\alpha-1}\D t
        +2\alpha \EE \Big[\langle M\rangle_{T}-\langle M\rangle_{\sft}\Big] \int_{T}^\infty(t-\sft)^{-2\alpha-1}\D t .\nonumber
\end{align}
Since $M$ is a square-integrable martingale we further have~$\EE[\langle M\rangle_t-\langle M\rangle_\sft] = \EE[M_t^2-M_{\sft}^2]=\EE\abs{M_t-M_{\sft}}^2 $. Hence we need to study the regularity of the process~$M_t=m_t(\widetilde{\bX}^{t,\bx})$ around~$\sft$. For $t<T$, $\bx,\bm{y}\in \Cc([t,T];\RR^d)$, the $\beta$-H\"older regularity of $\varphi$ combined with~\eqref{eq:est_reg_XxXy} and Jensen's inequality yields
\begin{align*}
        \abs{m_t(\bx)-m_t(\bm{y})} \lesssim \EE\abs{\bX^{t,\bx}_T-\bX^{t,\bm{y}}_T}^\beta 
        \lesssim \left(\abs{\bx(T)-\bm{y}(T)}^2 + \int_t^T (T-s)^{2H-1} \abs{\bx(s)-\bm{y}(s)}^2\ds\right)^{\beta/2}.
\end{align*}
In virtue of~\eqref{eq:est_reg_Xtildets} and Jensen's inequality, this entails
\begin{align}
        \EE\big\lvert m_t(\widetilde{\bX}^{t,\bx})-m_t(\bx)\big\lvert^2 
        &\lesssim \EE\big\lvert\widetilde{\bX}^{t,\bx}_T-\bx(T)\big\lvert^{2\beta} + \left(\int_t^T (T-s)^{2H-1}\EE \big\lvert\widetilde{\bX}^{t,\bx}_s-\bx(s)\big\lvert^2 \ds\right)^\beta \nonumber\\
        &\lesssim (t-\sft)^{\beta} (T-t)^{\beta(2H-1)} + \left(\int_t^T (T-s)^{2H-1} (t-\sft) (s-t)^{2H-1}\ds \right)^\beta \nonumber\\
        &\lesssim (t-\sft)^\beta (T-t)^{\beta(2H-1)}.
        \label{eq:est_mtXtilde}
\end{align}
Similarly, we have by~\eqref{eq:est_reg_XtXsft}    \begin{align}\label{eq:est_mtsftx}
        \abs{m_t(\bx)-m_\sft(\bx)} \lesssim \EE\abs{\bX^{t,\bx}_T-\bX^{\sft,\bx}_T}^\beta \lesssim (t-\sft)^{\beta/2} (T-t)^{\beta(H-1/2)}.
\end{align}
Therefore, combining \eqref{eq:est_mtXtilde} and~\eqref{eq:est_mtsftx} we derive that
\begin{align}
        \EE\abs{M_t-M_{\sft}}^2
        = \EE\big\lvert m_t(\widetilde{\bX}^{t,\bx})-m_{\sft}(\bx)\big\lvert ^2
        &\lesssim  \EE\big\lvert  m_t(\widetilde{\bX}^{t,\bx})-m_t(\bx)\big\lvert ^2 + \EE\abs{m_t(\bx)-m_{\sft}(\bx)}^2 \nonumber\\
        &\lesssim (t-\sft)^\beta (T-t)^{\beta(2H-1)}. 
        \label{eq:est_Mt_reg}
\end{align}
For the case~$t=T$ we simply notice that, thanks to the H\"older regularity of~$\bX^{\bx}$,
\begin{align}\label{eq:est_MT_reg}
    \EE\abs{M_T-M_{\sft}}^2 \le 2\EE\big\lvert\varphi(\bX^{\bx}_T)-\varphi(\bx(T))\big\lvert^2 +2 \EE\big\lvert\varphi(\bx(T))-\EE[\varphi(\bX^{\bx}_T)]\big\lvert^2 \lesssim (T-\sft)^{2\beta H}.
\end{align}
We thus reach the following conclusion in virtue of~\eqref{eq:rpz_Malpha}, \eqref{eq:est_Mt_reg}  and~\eqref{eq:est_MT_reg}
\begin{equation}\label{eq:powerlaw_estimate_Malpha}
\begin{aligned}
        \EE\abs{\mathscr{M}^\alpha_{\sft}}^2 
        &= 2\alpha \left(\int_{\sft}^T \EE\abs{M_t-M_{\sft}}^2 (t-\sft)^{-2\alpha-1}\D t
        + \EE\abs{M_T-M_{\sft}}^2 \int_T^\infty(t-\sft)^{-2\alpha-1}\D t\right)\\
        &\lesssim \int_{\sft}^T (T-t)^{\beta(2H-1)} (t-\sft)^{\beta-2\alpha-1}\D t + (T-\sft)^{2\beta H} (T-\sft)^{-2\alpha}
        \lesssim (T-\sft)^{2\beta H-2\alpha}
\end{aligned}
\end{equation}
where we used~$\beta\in(2\alpha,1]$. From~\eqref{eq:Malpha_eq_DalphaM} we obtain~$\mathscr{M}^\alpha_{\sft}=\scrD^\alpha_{T^-}(M)(\sft)$ which concludes the proof.
\end{proof}
Since~$\EE\abs{\mathscr{M}^\alpha_{\sft}}^2= \EE\left[\int_{\sft}^T (t-\sft)^{-2\alpha}\D \langle M\rangle_s \right]$, condition~\eqref{eq:fractional_condition} is met, we thus deduce from Theorem~\ref{thm:frac_IBP} that~$\scrD^\alpha_{T^-}(M)(\sft)\in L^2(\Omega)$ and, if $\varphi\in \Cc^1_{b,{\rm Lip}}(\RR^d)$, the fractional IBP holds.

\medskip

\underline{\emph{Step 3: H\"older continuous test functions}.} To conclude, we free ourselves from the restriction that~$\varphi$ is differentiable. Consider~$\varphi\in\Cc^\beta_{{\rm poly}}(\RR^d)$ and a sequence of functions~$(\varphi_n)_{n\in\NN}$ in~$\Cc^\beta_{{\rm poly}}(\RR^d)\cap\Cc^{1}_{b,{\rm Lip}}(\RR^d)$ approximating~$\varphi$ uniformly over compact subsets of~$\RR^d$ (which is possible thanks to the local version of the Stone--Weierestrass theorem). Analogously to the previous step, we define
\begin{align*}
        \Phi_n(\bx):=\EE[\varphi_n(\bX^{\bx}_T)],
        \quad M^n_t := \EE[\varphi_n(\bX^{\bx}_T)\lvert\Ff_t],\quad 
        m_t^n(\bx) := \EE[\varphi_n(\bX^{t,\bx}_T)].
\end{align*}
Hence the fractional IBP~\eqref{eq:frac_IBP_formula} holds for $D\Phi_n(\bx)(h)$ where~$h\in\Hh_\alpha$. We also introduce
\begin{align*}
        \Psi_\alpha(\bx;h) := \EE\left[\scrD^\alpha_{T^-}\big(M_T-M_{\cdot}\big)(\sft)\int_{\sft}^T (s-\sft)^\alpha \Big\langle\xi(s,\bX^{\bx}_s)h^*(s),\D W_s \Big\rangle\right].
\end{align*}
We will follow the structure of Step 3 of the proof of Theorem~\ref{th:main_IBP}. Noting that, by Cauchy--Schwarz inequality,
\begin{align*}
        \abs{D\Phi_n(\bx)(h)-\Psi_\alpha(\bx;h)}^2 \le \EE\abs{\scrD_{T^-}^\alpha\big(M^n_T-M^n_\cdot\big)(\sft)-\scrD_{T^-}^\alpha\big(M_T-M_\cdot\big)(\sft) }^2 \norm{\xi}^2_\infty \norm{h}^2_{\Hh_\alpha}, 
\end{align*}
we start by establishing the following estimate.
\begin{lemma}\label{lemma:est_Dalpha}
    Let $f\in\Cc^\beta_{{\rm poly}}(\RR^d)$, $\beta\in(0,1)$, and $M^f_t:=\EE[f(\bX^{\bx}_T)\lvert\Ff_t]$. Then for any $\alpha\in(0,\beta/2)$ there exist~$\delta\in(0,1/2)$ and~$p\in(1,2)$ such that
    \begin{align}
            \EE\abs{\scrD^\alpha_{T^-}\big(M^f_T-M^f_\cdot\big)(\sft)}^p \lesssim \Big(\EE\abs{f(\bX^{\bx}_T)}^2\Big)^{\delta p /2}.
    \end{align}
\end{lemma}
\begin{proof}
We leverage for this proof the results found in~\cite{samko1993fractional}. They are proved there for the left Riemann--Liouville integrals and derivatives and extend in a straightforward way to the right version used in this paper. We know from~\eqref{eq:Malpha_eq_DalphaM} that~$f(\bX^{\bx}_T)-M^f =\scrI^\alpha_{T^-}(\scrM^{f,\alpha}) / \Gamma(1-\alpha)$ where~$\scrM^{f,\alpha}_t:=\int_t^T (s-\sft)^{-\alpha} \D M^f_s$. Theorem 13.1 of \cite{samko1993fractional} states that if~$\scrM^{f,\alpha} \in L^1(\Tt)$ then we have the following representation, stated in Equation (13.5) of the same monograph:
\begin{align}\label{eq:rpz_Dalpha}
        \scrD^\alpha_{T^-}(M^f_T-M^f_\cdot)(\sft) = \frac{M^f_T-M^f_{\sft}}{\Gamma(1-\alpha)(T-\sft)^\alpha} + \frac{\alpha}{\Gamma(1-\alpha)} \int_{\sft}^T \frac{M^f_t-M^f_{\sft}}{(t-\sft)^{1+\alpha}}\dt.
\end{align}
In a very similar fashion to~\eqref{eq:powerlaw_estimate_Malpha}, we derive the following estimate
\begin{align*}
    \EE\left[\int_{\sft}^T \big\lvert\scrM^{f,\alpha}_t\big\lvert^2 \dt\right]
    &\lesssim \int_{\sft}^T (T-t)^{2(\beta H - \alpha)} \dt <\infty,
\end{align*}
since~$\alpha < \beta/2 <1/2$. 
In the last computation we used, for all $s\ge t$,
\begin{align*}
        \EE\big\lvert M^f_t-M^f_s\big\lvert^2
        &=\EE\big\lvert m_t^f(\widetilde{\bX}^{t,\bx}) -m_s^f(\widetilde{\bX}^{s,\bx}) \big\lvert^2\\
        &\lesssim \EE\big\lvert m^f_t(\widetilde{\bX}^{t,\bx}) - m^f_t(\widetilde{\bX}^{s,\bx})\big\lvert^2 + \EE\big\lvert m^f_t(\widetilde{\bX}^{s,\bx}) - m^f_s(\widetilde{\bX}^{s,\bx})\big\lvert^2\\
        &\lesssim \EE\big\lvert\widetilde{\bX}^{t,\bx}_T - \widetilde{\bX}^{s,\bx}_T\big\lvert^{2\beta} + \EE\big\lvert \widetilde{\bX}^{t,\bx}_T - \widetilde{\bX}^{s,\bx}_T\big\lvert^{2\beta} + \left(\int_t^T (T-r)^{2H-1} \EE\big\lvert \widetilde{\bX}^{t,\bx}_r - \widetilde{\bX}^{s,\bx}_r\big\lvert^{2}\dr\right)^{\beta/2} \\
        &\lesssim (s-t)^{\beta}(T-s)^{\beta(2H-1)},
\end{align*}
by Lemma~\ref{lemma:est_X} and where the computations are similar to~\eqref{eq:est_mtXtilde} and~\eqref{eq:est_mtsftx}. Since~$\int_{\sft}^T \lvert \scrM^\alpha_t \lvert \dt$ is finite in~$L^2(\Omega)$ it is also finite almost surely and hence the representation~\eqref{eq:rpz_Dalpha} holds almost surely.

    We are going to exploit the following estimates  $\EE\lvert M^f_{t}-M^f_{\sft}\lvert^2 \lesssim (t-\sft)^{\beta}(T-t)^{\beta(2H-1)}$ and $\EE\lvert M^f_{t}-M^f_{\sft}\lvert^2 \le 2\EE\lvert f(\bX^{\bx}_T)\lvert^2$.  Let us set~$\delta\in(0,1/2),\,p\in(1,2)$ for now and we will specify their values later. Two applications of Jensen's inequality to representation~\eqref{eq:rpz_Dalpha} yield
    \begin{align*}
        &\EE\abs{\scrD^\alpha_{T^-}\big(M^f_T-M^f_\cdot\big)(\sft)}^p \\
        &\lesssim 
        \frac{\Big(\EE \big\lvert M^f_{T}-M^f_{\sft}\big\lvert^2\Big)^{p/2}}{(T-\sft)^{p\alpha}} + (T-\sft)^{p-1} \int_{\sft}^T \frac{\Big(\EE \big\lvert M^f_{t}-M^f_{\sft}\big\lvert^2\Big)^{\delta p/2}\Big(\EE \big\lvert M^f_{t}-M^f_{\sft}\big\lvert^2\Big)^{(1-\delta)p/2}}{(t-\sft)^{p(1+\alpha)}}\dt \\
        &\lesssim \frac{\Big(2\EE\lvert f(\bX^{\bx}_T)\lvert^2\Big)^{p/2}}{(T-\sft)^{p\alpha}} +  \Big(\EE\lvert f(\bX^{\bx}_T)\lvert^2 \Big)^{\delta p /2} \int_{\sft}^T (t-\sft)^{p((1-\delta)\beta/2-1-\alpha)} (T-t)^{(1-\delta)p \beta (H-1/2)} \dt,
    \end{align*}
    and the integral is finite provided that~$p((1-\delta)\beta/2-1-\alpha)>-1$ and~$(1-\delta)p \beta (H-1/2)>-1$. The latter is guaranteed for all~$p<2$ since~$\beta(H-1/2)>-1/2$. Further set~$\delta \in(0, 1-2\alpha/\beta)$ such that~$(1-\delta)\beta/2-1-\alpha>-1$ and thus the former inequality is satisfied for any~$p\in(1, \frac{1}{1+\alpha-(1-\delta)\beta/2})$. The proof is complete.
\end{proof}
Moreover, for any~$f:\RR^d\to\RR$ with polynomial growth~$\abs{f(x)}\le C(1+\abs{x})^\kappa$ and any~$R>0,$ we have 
\begin{align}\label{eq:localisation_argument}
    \EE\Big[f(\bX^{\bx}_T)^2\one_{\abs{\bX^{\bx}_T}>R}\Big]
    \lesssim \left(\EE\left[ 1+ \abs{\bX^{\bx}_T}^{4\kappa}\right] \PP\big(1+\abs{\bX^{\bx}_T}^{4\kappa}>1+R^{4\kappa} \big) \right)^{1/2}
    \le \frac{\EE\left[ 1+ \abs{\bX^{\bx}_T}^{4\kappa}\right] }{1+R^{4\kappa}},
\end{align}
where we used Cauchy--Schwarz and Markov's inequalities. Further recall that~$\EE\abs{\bX^{\bx}_T}^{4\kappa}\lesssim \norm{\bx}^{4\kappa}_\infty$ such that~$\EE\Big[f(\bX^{\bx}_T)^2\one_{\abs{\bX^{\bx}_T}>R}\Big]$ tends to zero as~$R\nearrow+\infty$.  Let us pick the constants~$\delta\in(0,1/2)$ and~$p\in(1,2)$ obtained in Lemma~\ref{lemma:est_Dalpha}. Leveraging the latter, H\"older and BDG inequalities then yield 
\begin{align*}
        &\abs{D\Phi_n(\bx)(h)-\Psi_\alpha(\bx;h)} \\
        &\lesssim \Big(\EE\big\lvert \scrD^\alpha_{T^-}(M^{\varphi-\varphi_n}_T-M^{\varphi-\varphi_n}_\cdot)(\sft)\big\lvert^p \Big)^{1/p} \bigg(\EE\bigg\lvert \int_{\sft}^T (s-\sft)^{\alpha} \Big\langle\xi(s,\bX^{\bx}_s) h^*(s), \D W_s \Big\rangle \bigg\lvert^{\frac{p}{p-1}}\bigg)^{\frac{p-1}{p}}\\
        & \lesssim \left(\EE\abs{(\varphi-\varphi_n)(\bX^{\bx}_T)}^2\right)^{\delta/2}\norm{\xi}_{\infty} \left(\int_{\sft}^{T} (s-t)^{2\alpha} \abs{h^*(s)}^2\ds\right)^{1/2} 
\end{align*}
Recall that~$\norm{h}_{\Hh_\alpha}^2 = \int_{\sft}^{T} (s-t)^{2\alpha} \abs{h^*(s)}^2\ds$ and therefore, by the localisation argument~\eqref{eq:localisation_argument}, for any~$\ep>0$ there is~$R>0$ such that
\begin{align*}
        \abs{D\Phi_n(\bx)(h)-\Psi_\alpha(\bx;h)} 
        &\lesssim \left(\EE\Big[\abs{(\varphi-\varphi_n)(\bX^{\bx}_T)}^2\one_{\abs{\bX^{\bx}_T}\le R}\Big]+ \EE\Big[\abs{(\varphi-\varphi_n)(\bX^{\bx}_T)}^2\one_{\abs{\bX^{\bx}_T}>R}\Big]\right)^{\delta/2}\norm{h}_{\Hh_\alpha} \\
        &\le \Big(\sup_{\abs{x}\le R} \abs{\varphi(x)-\varphi_n(x)}^2+ \ep\Big)^{\delta/2} \norm{h}_{\Hh_\alpha},
\end{align*}
which goes to zero as~$n\nearrow+\infty$.  
Similarly, for any~$m,n\in\NN$, Lemma~\ref{lemma:est_Dalpha} applied to~$\Phi_n-\Phi_m$ yields for some~$\delta\in(0,1/2)$, as in~\eqref{eq:est_PhinPhi_m}, 
\begin{align*}
    \abs{(\Phi_n-\Phi_m)(\bx+h) - (\Phi_n-\Phi_m)(\bx)} 
    &\le \sup_{\lambda\in[0,1]}\norm{\big(\varphi_n-\varphi_m\big)(\bX^{\bx+\lambda h}_T)}_{L^2(\Omega)}^{\delta} \norm{\xi}_{\infty} \norm{h^*}_{L^2_\alpha}\\
    &\lesssim \Big(\sup_{\abs{x}\le R} \abs{\varphi_n(x)-\varphi_m(x)}+ \ep\Big)^{\delta} \norm{h}_{\Hh_\alpha},
\end{align*}
where the constant does not depend on~$n,m$. Taking $m$ to~$+\infty$ we deduce that for~$n$ large enough, $\sup_{\abs{x}\le R} \abs{\varphi_n(x)-\varphi_m(x)}\le \ep $. For $n\ge n_1 \vee n_2$ and $\norm{h}_{\Hh_\alpha}\le1$ we obtain
\begin{align*}
    \frac{\lvert\Phi(\bx+h)-\Phi(\bx)-\Psi_\alpha(\bx;h)\lvert}{\norm{h}_{\Hh_\alpha}}
    &\le  \frac{\lvert(\Phi-\Phi_n)(\bx+h)-(\Phi-\Phi_n)(\bx)\lvert}{\norm{h}_{\Hh_\alpha}} \\
    &\qquad+  \frac{\lvert D\Phi_n(\bx)h-\Psi_\alpha(\bx;h)\lvert}{\norm{h}_{\Hh_\alpha}}\\ 
    &\qquad+  \frac{\lvert\Phi_n(\bx+h)-\Phi_n(\bx)-D\Phi_n(\bx)h\lvert}{\norm{h}_{\Hh_\alpha}} \\
    &\le 2(2\ep)^\delta +  \frac{\lvert\Phi_n(\bx+h)-\Phi_n(\bx)-D\Phi_n(\bx)h\lvert}{\norm{h}_{\Hh_\alpha}} ,
\end{align*}
and since $\Phi_n\in\Dd^1_{\Hh_\alpha}(\Xx)$ for all $n\in\NN$, this proves that~$D\Phi(\bx)(h)=\Psi_\alpha(\bx;h)$. Finally, Cauchy--Schwarz inequality and Lemma~\ref{lemma:est_Dalpha} yield
 \begin{align*}
     \abs{\Psi_\alpha(\bx;h)} \lesssim \norm{\varphi(\bX^{\bx}_T)}_{L^2(\Omega)}^{\delta} \norm{h^*}_{L^2_\alpha}\lesssim \big(1+\norm{\bx}_\infty^{\delta}\big) \norm{h}_{\Hh_\alpha},
 \end{align*}
 which confirms the Fréchet differentiability and thereby ends the proof.

\subsection{Proof of Theorem \ref{thm:frac_IBP} C)}\label{sec:proof_fIBP_C}
Similarly to Corollary \ref{coro:reg_Phi} we have
\begin{align*}
        \abs{\Phi(\bx+h)-\Phi(\bx)}
        =\Big\lvert \int_0^1 D\Phi(\bx+\lambda h)(h)\D \lambda \Big\lvert.
\end{align*}
With a mild abuse of notations, let us denote~$M_t = \EE[\varphi(\bX^{\bx+\lambda h}_T)\lvert\Ff_t]$ in this proof. 
In virtue of Lemma~\ref{lemma:est_Dalpha}, with the parameters $p\in(1,2)$ and $\delta\in(0,1)$ obtained there, it follows from H\"older's and BDG inequalities that
\begin{align*}
    \abs{D\Phi(\bx+\lambda h)(h)}
    &\lesssim \big(\EE\abs{\scrD^\alpha_{T^-}(M_T-M_\cdot)(\sft)}^p\big)^{1/p} \left( \int_{\sft}^T (s-\sft)^{2\alpha} \abs{h^*(s)}^2 \ds \right)^{1/2}\\
    &\lesssim \left(\EE\abs{\varphi(\bX^{\bx+\lambda h}_T)}^2\right)^{\delta/2} \norm{h^*}_{L^2_\alpha} \lesssim \norm{h}_{\Hh_\alpha},
\end{align*}
since~$\varphi$ has polynomial growth. This yields the claim.
 

\section{A second order integration by parts}\label{sec:IBP_2}
In contrast with the standard BEL formula, the first IBP~\eqref{eq:main_IBP} does not feature the tangent process~$\bY^h$. This allows us to derive an integration by parts for the second derivative in a close fashion. We derive it in dimension one for simplicity and write $\sigma'$ for the derivative with respect to~$x$ instead of $\nabla\sigma$.
\begin{proposition}\label{prop:IBP_2}
    Let Assumptions~\ref{as:kernel} and~\ref{as:coefs} hold, $d=m=1$, $\gamma\in(0,H)$ and~$\phi\in\Bb(\Cc^\gamma_\infty)$.
    Furthermore, assume that~$x\mapsto\sigma'(t,x)$ is Lipschitz continuous uniformly in~$t$.     
    Then the map $\bx\mapsto\Phi(\bx)=\EE[\phi(\bX^{\bx})]$ belongs to~$\Dd^2_{\Hh}(\Cc^\gamma_\infty)$ and for all~$\bx\in\Cc^\gamma,\,g,h\in\Hh$ the integration by parts holds
    \begin{equation}\label{eq:main_IBP2}
    \begin{aligned}     
        D^2 \Phi(\bx)(g,h)
        &=\EE\left[\phi(\bX^{\bx})\int_\sft^T \xi(s,\bX_s^{\bx}) h^*(s) \D W_s  \int_\sft^T \xi(s,\bX_s^{\bx}) g^*(s)\D W_s \right]\\
        &\quad-\EE\left[ \phi(\bX^{\bx}) \int_\sft^T  \xi^2(s,\bX_s^{\bx})h^*(s) g^*(s)\ds\right].
    \end{aligned}
    \end{equation}
\end{proposition}
\begin{remark}
We recycle Remark \ref{rem:mainIBP} regarding the H\"older regularity condition.
\begin{enumerate}
    \item[i)] The first part of the proof (Steps 1 and 2) goes through with $\phi\in\Cc^2_{b,\Xx}(\Xx)$, entailing that $\Phi\in\Dd^2_{b,\Hh}(\Xx)$ and that the integration by parts formula holds for any~$\bx\in\Xx$.
    \item[ii)] It is also possible to prove~$\Phi\in \Dd^2_{\Hh}(\Xx)$ and to derive the IBP for all~$\bx\in\Xx$ under the condition that~$\phi(\bx)=\varphi(\bx(T))$ for some~$\varphi:\RR^d\to\RR$. The reason is that, since $\RR^d$ is locally compact, it is no longer necessary to appeal to the H\"older norm in Lemma~\ref{lemma:density_phi}.
\end{enumerate}
\end{remark}
The proof partially relies on the following estimate, which we prove at the end of the section.
\begin{lemma}\label{lemma:est_DeltahX}
    Let Assumptions~\ref{as:kernel} and~\ref{as:coefs} hold, and let~$\bx\in\Xx,h\in\Hh$. 
    The process $\Delta_h \bX^{\bx}=\bX^{\bx+h}-\bX^{\bx}$ is almost surely H\"older continuous and in particular~$\EE\norm{\bX^{\bx+h}-\bX^{\bx}}_{\infty}^2\lesssim \norm{h}^2_{\Hh}.$
\end{lemma}
\begin{proof}[Proof of Proposition \ref{prop:IBP_2}]
\emph{\underline{Step 0: Notations.}} 
We assume for the moment that $\phi\in\Cc^2_{b,\Cc^\gamma,{\rm Lip}}(\Cc^\gamma)$ and $\int_s^T K_H(r,s)^2 \abs{g^*(r)}^2\dr <\infty$ for all $s\in[\sft,T]$. 
We define for convenience~$\psi(\bx,z):=\phi(\bx)z$ for any~$\bx\in\Xx$ and~$z\in\RR$.  Denoting~$G^{\bx}:=\int_\sft^T \xi(s,\bX_s^{\bx})g^*(s)\D W_s$, Theorem~\ref{th:main_IBP} yields~$
    D\Phi(\bx)(g) 
    = \EE[\psi(\bX^{\bx},G^{\bx})]$.
    
\emph{\underline{Step 1: Derivatives.}} In this step we compute the Malliavin derivative and Fréchet derivative of this integrand. Let $\bx\in\Cc^\gamma_\infty$ such that~$\bX^{\bx}$ has~$\gamma$-H\"older paths. 
For any~$s\in[\sft,T]$, by the chain rule the Malliavin derivative reads
\begin{equation}
\begin{aligned}
\label{eq:MalliavinD_psi}
    \Df_s \psi(\bX^{\bx},G^{\bx}) &= D\phi(\bX^{\bx})(\Df_s \bX^{\bx}) G^{\bx} + \phi(\bX^{\bx})\xi(s,\bX_s^{\bx})g^*(s)\\
    &\quad - \phi(\bX^{\bx})\int_s^T \xi^2(r,\bX_r^{\bx}) \sigma'(r,\bX_r^{\bx})\Df_s\bX_r^{\bx} g^*(r) \D W_r.
\end{aligned}
\end{equation}
Lemma \ref{lemma:MalliavinD} entails that~$\EE\abs{\Df_s X_r}^2\le C \abs{K_H(r,s)}^2$ for a constant $C>0$, hence the stochastic integral is well-defined thanks to the ellipticity assumption \ref{as:coefs}:
\begin{align*}
    \EE\left(\int_s^T \xi^2(r,\bX_r^{\bx}) \sigma'(\bX_r^{\bx})\Df_s\bX_r^{\bx} g^*(r)  \D W_r\right)^2
    \lesssim \norm{\xi}^2_{\infty} \norm{\sigma'}_\infty \int_s^T \abs{K_H(r,s)}^2 \abs{g^*(r)}^2 \dr.
\end{align*}

On the other hand, denoting~$\partial_{\bx}$ for the Fréchet derivative with respect to~$\bx$, the chain rule yields
\begin{equation}
\begin{aligned}
\label{eq:derx_psi}
     \partial_{\bx} \psi(\bX^{\bx},G^{\bx})(h) 
    &= D\phi(\bX^{\bx})(\bY^{h})G^{\bx} + \phi(\bX^{\bx}) \partial_{\bx} G^{\bx}(h)\\
    &= D\phi(\bX^{\bx})(\bY^{h})G^{\bx} - \phi(\bX^{\bx})\int_\sft^T \xi^2(r,\bX_r^{\bx}) \sigma'(r,\bX_r^{\bx})\bY^{h}_r g^*(r) \D W_r.
\end{aligned}
\end{equation}
We computed $\partial_{\bx} G^{\bx}(h)$ as a derivative in $L^p(\Omega;\Xx)$ similarly to Proposition~\ref{prop:Frechet}. 
Leveraging that $\xi,\sigma'$ are bounded, that $\sigma'(r,\cdot)$ is Lipschitz continuous and hence so is $\xi(r,\cdot)$, we obtain~$\partial_{\bx} G^{\bx}(h)=\int_{\sft}^T \partial_{\bx}\xi(s,\bX^{\bx}_s)g^*(s)\D W_s$. The chain rule yields~\eqref{eq:derx_psi}.

\emph{\underline{Step 2: The integration by parts.}} 
We plug the formula from Lemma~\ref{lemma:equality_Y_DX}
\begin{align*}
    \bY^{h}_r= \int_\sft^T \Df_s \bX_r^{\bx} \xi(s,\bX_s^{\bx})h^*(s)\ds
\end{align*}
into the representation~\eqref{eq:derx_psi} of~$\partial_{\bx} \psi(\bX^{\bx},G^{\bx})(h)$ in order to relate it with~$\Df_s \psi(\bX^{\bx},G^{\bx})$. Applying stochastic Fubini theorem \cite[Theorem 65]{protter2005stochastic} and comparing with~\eqref{eq:MalliavinD_psi} yields  
\begin{align*}
    \partial_{\bx} \psi(\bX^{\bx},G^{\bx})(h) 
    &= \int_\sft^T \bigg\{  D\phi(\bX^{\bx})\Big(\Df_s \bX^{\bx} \xi(s,\bX_s^{\bx})h^*(s)\Big) G^{\bx} \\
    &\quad -\phi(\bX^{\bx})\int_s^T \xi^2(r,\bX_r^{\bx}) \sigma'(\bX_r^{\bx}) \Big( \Df_s \bX_r^{\bx} \xi(s,\bX_s^{\bx})h^*(s)\Big) g^*(r) \D W_r \bigg\} \ds \\
    &= \int_\sft^T \Big\{\Df_s \psi(\bX^{\bx},G^{\bx}) - \phi(\bX^{\bx}) \xi(s,\bX_s^{\bx})g^*(s) \Big\}\xi(s,\bX_s^{\bx}) h^*(s)\ds.
\end{align*}
We conclude by taking the derivative inside the expectation by dominated convergence and applying the Malliavin integration by parts as follows
\begin{align*}
    & D^2\Phi(\bx) (g,h)
    =\EE\left[\partial_{\bx} \psi(\bX^{\bx},G^{\bx})(h) \right]\\
    &=  \EE\left[\int_\sft^T  \Df_s \psi(\bX^{\bx},G^{\bx}) \xi(s,\bX_s^{\bx}) h^*(s)\ds\right] 
    - \EE\left[\int_\sft^T \phi(\bX^{\bx})  \xi^2(s,\bX_s^{\bx})g^*(s) h^*(s)\ds\right] \\
    &= \EE\left[\psi(\bX^{\bx},G^{\bx})\int_\sft^T \xi(s,\bX^{\bx}_s) h^*(s) \D W_s
    - \phi(\bX^{\bx}) \int_\sft^T \xi^2(s,\bX^{\bx}_s)g^*(s)  h^*(s)\ds\right]\\
    &=: \Psi_2(\bx,g,h).
\end{align*}
Cauchy--Schwarz inequality entails~$\abs{\Psi_2(\bx,g,h)}\lesssim \norm{g}_{\Hh}\norm{h}_{\Hh}$.

\emph{\underline{Step 3: $g\in\Hh$.}} Let $g\in\Hh$ and a sequence $(g_n)_{n\in\NN}\subset\Hh$ such that~$\int_s^T K_H(r,s)^2 \abs{g_n^*(s)}^2\ds<\infty$ for all $s\in\Tt,\,n\in\NN$ and~$\norm{g^*-g_n^*}_{L^2(\Tt)}$ goes to zero  (e.g. $(g_n)_n$ is a bounded sequence). For each~$n\in\NN$, the previous step guarantees that the Fréchet derivative~$D^2\Phi(\bx)(g_n,h)$ exists which means
\begin{align*}
    \lim_{\norm{h}_{\Hh}\searrow0} \frac{\abs{D\Phi(\bx+h)(g_n)-D\Phi(\bx)(g_n)-\Psi_2(\bx;g_n,h)}}{\norm{h}_{\Hh}} = 0.
\end{align*}
Moreover, Cauchy--Schwarz inequality and the boundedness of~$\phi$ and~$\xi$ imply that, for any~$n\in\NN$,
\begin{align*}
    \abs{\Psi_2(\bx;g_n,h) - \Psi_2(\bx;g,h)}
    \lesssim \norm{h^*}_{L^2(\Tt)} \norm{g^*-g_n^*}_{L^2(\Tt)}.
\end{align*}
The Lipschitz continuity and boundedness of~$\phi$ and~$\xi(t,\cdot)$, again combined with Cauchy--Schwarz inequality, further entail
\begin{align*}
    \abs{D\Phi(\bx+h)(g-g_n) - D\Phi(\bx)(g-g_n)} 
    &\le \EE\left[ \abs{\phi(\bX^{\bx+h})-\phi(\bX^{\bx})} \Big\lvert\int_{\sft}^T \xi(s,\bX^{\bx+h}_s) (g^*-g_n^*)(s)\D W_s\Big\lvert\right]\\
    &\quad + \EE\left[ \abs{\phi(\bX^{\bx})} \Big\lvert\int_{\sft}^T \big(\xi(s,\bX^{\bx+h}_s)-\xi(s,\bX^{\bx}_s)\big) (g^*-g_n^*)(s)\D W_s\Big\lvert\right]\\
    &\lesssim \left(\EE\norm{\bX^{\bx+h}-\bX^{\bx}}_{\infty}^2 \right)^{\half} \norm{g^*-g_n^*}_{L^2(\Tt)}\\
    &\lesssim \norm{h}_{\Hh}  \norm{g^*-g_n^*}_{L^2(\Tt)},
\end{align*}
where we concluded thanks to Lemma~\ref{lemma:est_DeltahX}. Finally, we obtain that
\begin{align*}
    \frac{\abs{D\Phi(\bx+h)(g)-D\Phi(\bx)(g)-\Psi_2(\bx;g,h)}}{\norm{h}_{\Hh}}
    &\le \frac{\abs{D\Phi(\bx+h)(g-g_n) - D\Phi(\bx)(g-g_n)} }{\norm{h}_{\Hh}}\\
    &\quad + \frac{\abs{D\Phi(\bx+h)(g_n)-D\Phi(\bx)(g_n)-\Psi_2(\bx;g_n,h)}}{\norm{h}_{\Hh}}\\
    &\quad + \frac{\abs{\Psi_2(\bx;g_n,h) - \Psi_2(\bx;g,h)}}{\norm{h}_{\Hh}}\\
    &\lesssim \norm{g^*-g_n^*}_{L^2(\Tt)},
\end{align*}
confirming that~$D^2\Phi(\bx)(g,h) = \Psi_2(\bx;g,h)$.

\emph{\underline{Step 4: General test functions.}} Let us finally extend this result to~$\phi\in\Bb^\varrho(\Cc^\gamma_\infty)$. For that purpose, consider the sequence~$(\phi_n)_{n\in\NN}\subset\Cc^1_{b,\Cc^\gamma,{\rm Lip}}(\Cc^\gamma_\infty)$ constructed in Lemma~\ref{lemma:density_phi} and for all~$\bx\in\Cc^\gamma,h\in\Hh$, (re)define~$\Phi_n(\bx)(h)=\EE[\phi_n(\bX^{\bx})]$. Analogously to \eqref{eq:est_PhinPhi_m}, Taylor's formula and Cauchy--Schwarz inquality yield
\begin{align*}
    \abs{D(\Phi_n-\Phi_m)(\bx+h)(g)-D(\Phi_n-\Phi_m)(\bx)(g)} \le \sup_{\lambda\in[0,1]}\norm{(\phi_n-\phi_m)(\bX^{\bx+\lambda h})}_{L^2(\Omega)} \norm{\xi}_\infty^2 \norm{g}_{\Hh}\norm{h}_{\Hh},
\end{align*}
and, similarly to~\eqref{eq:cvg_IBP} it also holds
\begin{align*}
    \abs{D^2\Phi_n(\bx)(g,h) - \Psi_2(\bx;g,h)}\le \norm{\phi_n(\bX^{\bx})-\phi(\bX^{\bx})}_{L^2(\Omega)} \norm{\xi}_\infty^2 \norm{g}_{\Hh}\norm{h}_{\Hh}.
\end{align*}
Eventually this implies, by taking the limit as $m$ to $+\infty$ and for $n$ large enough,
\begin{align*}
    \frac{1}{\norm{h}_{\Hh}} \abs{D\Phi(\bx+h)(g) - D\Phi(\bx)(g) - \Psi_2(\bx;g,h)}
    &\le \frac{1}{\norm{h}_{\Hh}} \Big(\abs{D(\Phi-\Phi_n)(\bx+h)(g)-D(\Phi-\Phi_n)(\bx)(g)} \\
    &\quad + \abs{D^2\Phi_n(\bx)(g,h) - \Psi_2(\bx;g,h)}\\
    &\quad + \abs{D\Phi_n(\bx+h)(g) - D\Phi_n(\bx)(g) - D^2\Phi_n(\bx)(g,h)}\Big)
\end{align*}
tends to zero as~$\norm{h}_{\Hh}$ does, thereby proving that~$\Phi\in\Dd^2_{\Hh}(\Cc^\gamma_\infty)$. The boundedness follows from the Cauchy--Schwarz inequality.
\end{proof}


\begin{proof}[Proof of Lemma \ref{lemma:est_DeltahX}]
    Let $\sft\le s<t\le T$. Note that we have
    \begin{align*}
        \Delta_h  \bX^{\bx}_t-\Delta_h  \bX^{\bx}_s
        = h(t)-h(s) 
        &+  \int_{\sft}^s \big(K_b(t,r)-K_b(s,r)\big) \big( b(r,\bX^{\bx+h}_r) - b(r,\bX^{\bx}_r) \big) \dr\\
        &+ \int_s^t K_b(t,r) \big( b(r,\bX^{\bx+h}_r) - b(r,\bX^{\bx}_r) \big) \dr\\     
        &+ \int_{\sft}^s \big(K_\sigma(t,r)-K_\sigma(s,r)\big) \big( \sigma(r,\bX^{\bx+h}_r) - \sigma(r,\bX^{\bx}_r) \big) \D W_r \\
        &+ \int_s^t K_\sigma(t,r) \big( \sigma(r,\bX^{\bx+h}_r) - \sigma(r,\bX^{\bx}_r) \big) \D W_r .
    \end{align*}
    Leveraging standard computations similar to~\eqref{eq:comps_estimate_DeltahX} and~\eqref{eq:trick_Lp_bound}, involving Jensen and BDG inequalities as well as the Lipschitz continuity of~$b,\sigma$, we obtain for all~$p\ge2$
    \begin{align*}
        &\EE\abs{ \Delta_h  \bX^{\bx}_t-\Delta_h  \bX^{\bx}_s}^p \\
        &\lesssim \abs{h(t)-h(s)}^p 
        + \sup_{r\in\Tt} \EE\abs{\Delta_h  \bX^{\bx}_r}^p \left(\left(\int_0^s \big(K_H(t,r)-K_H(s,r)\big)^2 \dr\right)^{p/2} + \left(\int_s^t K_H(t,r)^2\dr\right)^{p/2}  \right)\\
        &\lesssim \norm{h}_{\Hh}^p (t-s)^{Hp} + \norm{h}_\infty^p(t-s)^{Hp},
    \end{align*}
    where we concluded with~\eqref{eq:HhinCH}, Assumption~\ref{as:kernel} and~\eqref{eq:bound_DeltaX} and we further note that~$\norm{h}_\infty\lesssim \norm{h}_{\Hh}$. Therefore Kolmogorov continuity theorem~\cite[Chapter 1, Theorem 2.1]{revuz2013continuous} yields 
    \begin{align*}
        \EE\left[\sup_{s\neq t}\frac{\abs{ \Delta_h  \bX^{\bx}_t-\Delta_h  \bX^{\bx}_s}^p}{\abs{t-s}^{p\gamma} }\right]\lesssim \norm{h}_{\Hh},
    \end{align*}
    for all~$\gamma\in(0,H)$,
    and this yields the claim by setting $s=0$.
\end{proof}


\section{Application to stochastic (rough) volatility models}\label{sec:SV}

In this section we modify slightly the underlying model to give a self-contained proof of an integration by parts formula for stochastic volatility models where the volatility factor satisfies an SVE. These include rough volatility models such as rough Bergomi and rough Heston (modulo non-trivial smoothness assumptions on the coefficients) and a  subclass of forward variance models. 

We are chiefly interested in differentiating an option price~$\EE[\phi(X^{\bv})\lvert\Ff_t]$ with respect to the initial variance curve $\bv$ (see \eqref{eq:roughvolmodel} below). Such functional derivatives arise in the dynamics of the option price (when applying a functional Itô formula for instance) and are central in the asymptotic expansion~\cite{bergomi2012stochastic} or the hedging formula~\cite{viens2019martingale}. As mentioned before, our formulae are not directly applicable to this context because they only hold when the direction lies in~$\Hh$ or $\Hh_\alpha$. Nevertheless, this derivative is a non-standard type of Greek which represents the sensitivity with respect to changes in the forward variance curve and, as such, in the variance swap volatility~$\hat\sigma_{VS}(\sft,T)=\sqrt{\frac{1}{T} \int_\sft^T \bv_s\ds}$ or in the VIX.  

For this purpose we introduce $\overline{W}$, $W$ and $\bm{W}:=(\overline{W},W)^\top$ as $\RR^d$, $\RR$ and~$\RR^{d+1}$-valued Brownian motions, respectively. Further define the~$\RR^{d}$-valued Brownian motion $B:=\bm{\rho} \bm{W}=\overline{\rho}\overline{W}+\rho W$, where the correlation matrix is~$\bm{\rho}:=(\overline{\rho},\rho)$ with~$\overline{\rho}\in\RR^{d\times d}$ and $\rho\in\RR^d$. 
We then consider the following model, a couple~$(X^{\bv},V^{\bv})$ where $X^{\bv}$ plays the role of the log-price ($X^{\bv}=\log(S^{\bv}_t/S_0)$ wher $S^{\bv}$ was defined in~\eqref{eq:SV_model_intro}) and $V^{\bv}$ of a stochastic factor solution to an SVE:
\begin{equation}\label{eq:roughvolmodel}
\begin{aligned}
    &X_t^{\bv}:=\int_\sft^t \zeta(s,V^{\bv}_s)\D W_s - \half \int_\sft^t \zeta^2(s,V^{\bv}_s) \ds, \\
    &V^{\bv}_s = \bv_s + \int_\sft^s K_b(s,r)b(r,V^{\bv}_r)\dr + \int_\sft^s K_\sigma(s,r) \sigma(r,V^{\bv}_r)\D B_r,
\end{aligned}
\end{equation}
where $\bv\in\Xx$, $b,\sigma$ and $K_b,K_\sigma$ are the same as in Equation~\eqref{eq:main_SVE} and satisfy Assumption~\ref{as:kernel} and~\ref{as:coefs}, respectively. Finally, $\zeta:\Tt\times\RR^d\to\RR$ is a measurable function. The main result of this section and its proof follow closely Theorem~\ref{th:main_IBP}. The novelty required to match tangent process and Malliavin derivative is to differentiate only with respect to the Brownian motion~$\overline{W}$. 
\begin{proposition}\label{prop:IBP_roughvol}
    Let Assumptions~\ref{as:kernel}  and~\ref{as:coefs} hold, $\zeta\in\Cc^1_b(\RR^d)$, $\gamma\in(0,H)$, $\phi\in \Bb^{\varrho}(\Cc^\gamma_\infty)$ and~$\overline{\rho}$ be invertible. Then the mapping $ \bv\mapsto\Phi(\bv):=\EE[\phi(X^{\bv})]$ belongs to~$\Dd^1_{\Hh}(\Cc^\gamma_\infty)$ and for all~$\bv\in\Cc^\gamma,\,h\in\Hh$, the Bismut--Elworthy--Li formula holds 
    \begin{equation}\label{eq:rvol_IBP}  
        D \Phi(\bv)(h)
        = \EE\left[\phi(X^{\bv})\int_\sft^T \Big\langle(\overline{\rho})^{-1} \xi(r,V_r^{\bv})h^*(r),\D \overline{W}_r\Big\rangle\right].
    \end{equation}
\end{proposition}
\begin{remark}\
\begin{enumerate}
    \item[i)]  In the common one-dimensional case~$d=1$ we require that~$\overline{\rho}:=\sqrt{1-\rho^2}\neq0$, or equivalently that the correlation~$\D\langle W,B\rangle_t = \rho\in(-1,1)$.
    \item[ii)]  Once again we can work on $\Xx$ instead of $\Cc^\gamma$ provided $\phi$ only acts on $X_T^{\bv}$ or $\phi\in\Cc^1_{b,\Xx}(\Xx)$, see Remark~\ref{rem:mainIBP}.  
    \item[iii)] We expect the fractional IBP (Theorem~\ref{thm:frac_IBP}) and the second order IBP (Proposition~\ref{prop:IBP_2}) to follow suit in this model as well. As can be seen below, the proof is very similar to Theorem~\ref{th:main_IBP}.
\end{enumerate}
\end{remark}
\begin{proof}
Assume for now that $\phi\in\Cc^1_{b,\Cc^\gamma,{\rm Lip}}(\Cc^\gamma)$. 
We denote~$\Df^{\overline{W}}_r B_s = \big( \Df^{\overline{W}^{(j)}}_r B_s^{(i)}\big)_{i,j=1}^d$ the $\RR^{d\times d}$ Malliavin derivative with respect to the Brownian motion~$\overline{W}$ and we observe that~$\Df^{\overline{W}}_r B_s=\overline{\rho} \one_{r<s}$. 

Let us denote by~$Y^h,\,Z^h$ the derivatives of~$\bv\mapsto X^{\bv}$ and~$\bv\mapsto V^{\bv}$, respectively, in the direction~$h\in\Hh$. Identically to the analysis presented in  Lemma~\ref{lemma:equality_Y_DX} we observe that~$Z_t^{h}$ is equal to~$\int_{\sft}^T \Df^{\overline{W}}_s V_t^{\bv} (\sigma(s,V^{\bv}_s)\overline{\rho})^{-1}h^*(s) \ds$
as they both solve the SVE
\begin{align*}
    Z_t = \int_{\sft}^t K_\sigma(t,s)h^*(s) \,\ds + \int_{\sft}^t K_b(t,s)\nabla b(s,V_s^{\bv}) Z_s \ds + \int_{\sft}^t K_\sigma(t,s)\nabla \sigma^{(j)}(s,V_s^{\bv}) Z_s \D B_s^{(j)}.
\end{align*}
As a consequence, we deduce that the Fréchet and Malliavin derivatives of the log-price also match since $\overline{W}$ and $W$ are independent:
\begin{align*}
    Y_t^{h} 
    &= \int_\sft^t \nabla \zeta(s,V_s^{\bv}) Z_s^{h} \D W_s - \half \int_\sft^t \nabla \zeta^2(s,V_s) Z_s^{h} \ds
    =\int_{\sft}^T \Df^{\overline{W}}_s X_t^{\bv}(\sigma(s,V^{\bv}_s)\overline{\rho})^{-1} h^*(s)\ds .
\end{align*}
Finally we obtain the integration by parts formula as in Equation~\eqref{eq:IBP_computation}
\begin{align*}
    D \Phi(\bv)(h)
    = \EE\left[D\phi(X^{\bv})(Y^h) \right]
    & = \EE\left[\int_\sft^T D \phi(X^{\bv})( \Df^{\overline{W}}_r X^{\bv}) \big(\sigma(r,V_r^{\bv})\overline{\rho}\big)^{-1} h^*(r)\dr\right]\\
    &    = \EE\left[\phi(X^{\bv}) \int_\sft^T \Big\langle(\overline{\rho})^{-1} \xi(r,V_r^{\bv}) h^*(r) ,\D \overline{W}_r\Big\rangle\right].
\end{align*}
This result extends to $\phi\in\Bb^\varrho(\Cc^\gamma_\infty)$ in exactly the same way as Step 2 of the proof of Theorem~\ref{th:main_IBP}.
\end{proof}

\appendix
\section{Technical proofs}\label{sec:appendix}
\subsection{Proof of Proposition \ref{prop:Frechet}}\label{subsec:proof_PropFrechet}
The proof relies on the following estimate.
\begin{lemma}\label{lemma:cty_Z}
   Let Assumptions \ref{as:kernel} and~\ref{as:coefs} hold.
    Let~$h\in\Xx$. For all~$p\ge2$, there is a constant~$C>0$ such that for all~$\sft\le t'\le t\le T$,
    \begin{align*}
        \EE\abs{\bZ^{h}_t-\bZ^{h}_{t'}}^p \le C \norm{h}_{\infty}^{2p} (t-t')^{H(p-2)}.
    \end{align*}
\end{lemma}
\begin{proof}
\emph{\underline{Step 1.}} For all~$t\in\Tt$ we denote~$\Delta_h \bX^{\bx}_t:=\bX^{\bx+h}_t - \bX^{\bx}_t$. 
By Lipschitz continuity of~$b,\sigma$, Jensen, BDG and H\"older's inequalities yield
for all~$t\in\Tt$ and all~$p\ge 2$,
\begin{equation}\label{eq:comps_estimate_DeltahX}
\begin{aligned}
    \EE\abs{\Delta_h \bX^{\bx}_t}^p
    &\lesssim  \abs{h(t)}^p + \EE\left(\int_\sft^t \abs{K_b(t,s)}\abs{b(s,\bX^{\bx+h}_s)-b(s,\bX^{\bx}_s)} \ds\right)^{p}\\
    &\qquad +  \EE\left(\int_\sft^t \abs{K_\sigma(t,s)}^2\abs{\sigma(s,\bX^{\bx+h}_s)-\sigma(s,\bX^{\bx}_s)}^2 \ds\right)^{p/2}\\
    &\lesssim \abs{h(t)}^p + \int_\sft^t \abs{K_b(t,s)}\EE\abs{b(s,\bX^{\bx+h}_s)-b(s,\bX^{\bx}_s)}^p \ds \\
    &\qquad + \int_\sft^t \abs{K_\sigma(t,s)}^2\EE\abs{\sigma(s,\bX^{\bx+h}_s)-\sigma(s,\bX^{\bx}_s)}^p \ds\\
    &\lesssim \abs{h(t)}^p + \int_\sft^t \abs{\KH(t,s)}^2\EE\abs{\Delta_h \bX^{\bx}_s}^p\ds.
\end{aligned}
\end{equation}
A similar computation holds for the drift term. In the second inequality, we used H\"older's inequality in the following fashion, where for any process~$F$:
\begin{align}\label{eq:trick_Lp_bound}
    \EE\left(\int_{\sft}^t \abs{\Ks(t,s)}^2 \abs{F(s)}^2 \ds\right)^{p/2} 
    &= \EE\left(\int_{\sft}^t \abs{\Ks(t,s)}^{2-4/p}\abs{\Ks(t,s)}^{4/p}  \abs{F(s)}^2 \ds\right)^{p/2} \\
    &\le \left(\int_{\sft}^t \abs{\Ks(t,s)}^{2}\ds \right)^{\frac{p-2}{2}} \EE\int_{\sft}^t \abs{\Ks(t,s)}^2\abs{F(s)}^p\ds.
    \nonumber
\end{align}
For a constant~$C>0$, recall that~$R_{C,H}$ \eqref{eq:Res_powerlaw} is the resolvent of~$C\KH^2$. The generalised Gr\"onwall lemma (Lemma~\ref{lemma:Grownall}) thus entails
\begin{align}\label{eq:bound_DeltaX}
    \EE\abs{\Delta_h \bX^{\bx}_t}^p \lesssim \abs{h(t)}^p + \int_r^t R_{C,H}(t,s) \abs{h(s)}^p \ds \lesssim \norm{h}^p_{\infty}.
\end{align}
\emph{\underline{Step 2.}} For~$\chi=b,\sigma^{(j)},j=1,\cdots, m$, denote
\begin{align*}
    \widetilde{\chi}^{ \bx}_h(s):=\int_0^1 \nabla \chi\big(s,\bX^{\bx}_s + \lambda\Delta_h \bX^{\bx}_t\big)\D \lambda    
\end{align*}
such that~$\chi(s,\bX^{ \bx+h}_s)-\chi(s,\bX^{\bx}_s)=\widetilde{\chi}^{ \bx}_h(s)\Delta_h \bX^{\bx}_s$ and by Lipschitz continuity, \begin{align}\label{eq:delta_f}
    \abs{\nabla \chi(s,\bX_s^{\bx})-\widetilde{\chi}^{ \bx}_h(s)}\lesssim \abs{\Delta_h \bX^{\bx}_s}.
\end{align} Since~$t,t',\bx,h$ are fixed, we omit them from the sub- and superscripts for the remainder of the proof. With the tools recently introduced we have
\begin{align*}
    \bZ_t &= \Delta\bX_t - \bY_t\\
    &=\int_{\sft}^t K_b(t,s) \Big(\widetilde{b}(s)\Delta \bX_s - \nabla b(s,\bX_s)\bY_s\Big)\ds + \int_{\sft}^t K_\sigma(t,s) \Big(\widetilde{\sigma}^{(j)}(s)\Delta \bX_s - \nabla \sigma^{(j)}(s,\bX_s)\bY_s\Big)\D W^{(j)}_s.
\end{align*}
With~$t'<t$ this entails for all~$p\ge2$,
\begin{align*}
    \EE\abs{\bZ_t-\bZ_{t'}}^p 
    &\lesssim \EE\abs{\int_{t'}^t K_b(t,s) \Big(\widetilde{b}(s)\Delta \bX_s - \nabla b(s,\bX_s)\bY_s\Big)\ds}^p \\
    &+ \EE\abs{\int_{t'}^t K_\sigma(t,s) \Big(\widetilde{\sigma}^{(j)}(s)\Delta \bX_s - \nabla \sigma^{(j)}(s,\bX_s)\bY_s\Big)\D W^{(j)}_s}^p\\
    & + \EE\abs{\int_{\sft}^{t'} \big(K_b(t,s)-K_b(t',s)\big) \Big(\widetilde{b}(s)\Delta \bX_s - \nabla b(s,\bX_s)\bY_s\Big)\ds}^p\\
    &+ \EE\abs{\int_{\sft}^{t'} \big(K_\sigma(t,s)-K_\sigma(t',s)\big) \Big(\widetilde{\sigma}^{(j)}(s)\Delta \bX_s - \nabla \sigma^{(j)}(s,\bX_s)\bY_s\Big)\D W^{(j)}_s}^p\\
    &=: \bm{A}_1 + \bm{A}_2 + \bm{A}_3 +\bm{A}_4.
\end{align*}
We deal with~$ \bm{A}_2$ and~$ \bm{A_4}$ and note that~$ \bm{A}_1$ and~$ \bm{A}_3$ are dealt with in the same way, using Jensen's inequality instead of BDG inequality. Combined with H\"older's inequality and~\eqref{eq:delta_f}, the latter yields
\begin{align*}
    \bm{A}_2 
    &\lesssim \EE\left(\int_{t'}^t \abs{K_\sigma(t,s)}^2 \abs{\widetilde{\sigma}^{(j)}(s)-\nabla\sigma^{(j)}(s,\bX_s)}^2 \abs{\Delta \bX_s}^2 \ds \right)^{p/2}\\
    &\quad + \EE\left(\int_{t'}^t \abs{K_\sigma(t,s)}^2 \abs{\nabla\sigma^{(j)}(s,\bX_s)}^2 \abs{\Delta \bX_s-\bY_s}^2 \ds \right)^{p/2}\\
    &\lesssim \left(\int_{t'}^t \abs{\KH(t,s)}^{2}\ds \right)^{\frac{p-2}{2}} \EE\int_{t'}^t \abs{\KH(t,s)}^2 \Big( \abs{\Delta \bX_s}^{2p} + \abs{\bZ_s}^p\Big) \ds. 
\end{align*}
Similarly, exploiting~$\abs{\Ks(t,s)-\Ks(t',s)}\le \abs{\KH(t',s)}$, we get
\begin{align*}
    \bm{A}_4 \lesssim \left(\int_{\sft}^{t'} \abs{\Ks(t,s)-\Ks(t',s)}^{2}\ds \right)^{\frac{p-2}{2}} \EE \int_{\sft}^{t'} \abs{\KH(t',s)}^2 \Big( \abs{\Delta \bX_s}^{2p} + \abs{\bZ_s}^p\Big) \ds.
\end{align*}
In particular, for $t'=\sft$ Equation \eqref{eq:bound_DeltaX} entails
\begin{align*}
    \EE\abs{\bZ_t}^p \lesssim\int_{\sft}^t \abs{\KH(t,s)}^2\Big(\norm{h}_{\infty}^{2p} + \EE\abs{\bZ_s}^p\Big) \ds,
\end{align*}
which yields by Lemma~\ref{lemma:Grownall} (item 3) and the integrability of~$R_{C,H}$ that~$\EE\abs{\bZ_t}^p \lesssim  \norm{h}_{\infty}^{2p}.$
Gathering all the estimates together and leveraging Assumption~\ref{as:kernel} we obtain
\begin{align*}
    \EE\abs{\bZ_t-\bZ_{t'}}^p 
    &\lesssim \norm{h}_{\infty}^{2p}\left[\left(\int_{t'}^t \KH(t,s)^2\ds\right)^{\frac{p-2}{2}} + \left(\int_{\sft}^{t'} \abs{K_H(t,s)-K_H(t',s)}\ds\right)^{\frac{p-2}{2}}\right]\\
    &\lesssim \norm{h}_{\infty}^{2p} (t-t')^{H(p-2)},
\end{align*}
which yields the claim.
\end{proof}

\begin{proof}[Proof of Proposition~\ref{prop:Frechet}]
    Lemma~\ref{lemma:cty_Z} and Kolmogorov continuity criterion, see e.g. \cite[Chapter 1, Theorem 2.1]{revuz2013continuous}, yield
    \begin{align*}
        \EE\left[\sup_{\sft\le t\neq t'\le T} \left(\frac{\bZ_t^{h}-\bZ_{t'}^{h}}{\abs{t-t'}^\alpha} \right)^p \right] \le C' \norm{h}^{2p}_{\infty,r},
    \end{align*}
    for all~$\alpha\in[0,H-(2H+1)/p)$, $p>2+1/H$ and where~$C'>0$ depends only on $C,p$ and~$\alpha$. Setting $t'=\sft$ gives in particular~$\EE\norm{\bZ^{h}}_{\infty}^p\lesssim \norm{h}^{2p}_{\infty}$. Considering a sequence~$(h^i)_{i\in\NN}$ converging to zero, the claim is proved. 
\end{proof}

\subsection{Proof of Lemma \ref{lemma:MalliavinD}}\label{subsec:proof_DX}
\begin{proof}
    This is a straightforward adaptation of~\cite[Theorem 2.2.1]{nualart2006malliavin}. Let us fix~$\sft,\bx$. Consider the Picard approximations given by~$\bX^0_t=\bx(t)$ and for $n\in\NN$:
    \begin{align*}
        \bX^{n+1}_t= \bx(t) + \int_\sft^t K_b(t,s)b(s,\bX^n_s)\ds + \int_\sft^t K_\sigma(t,s)\sigma(s,\bX^n_s)\D W_s.
    \end{align*}
    By induction, one can prove that for all~$n\in\NN$, $\bX^n_t\in\DD^{1,2}$ for all~$0\le r< t\le T$, $ \psi_n(r,t) :=
        \EE[\abs{\Df_r \bX^n_t}^2] <\infty$
    and for some constants~$c_1,c_2>0$,
    \begin{align}\label{eq:psi_ineq}
        \psi_{n+1}(r,t) \le c_1 \abs{K_\sigma(t,r)}^2 + c_2 \int_r^t \KH(t,s)^2\psi_n(r,s)\ds.
    \end{align}
    These properties can be deduced from the equality
    \begin{align*}
        \Df_r \bX_t^{n+1} = \int_r^t K_b(t,s) \nabla b(s,\bX_s^n) \Df_r \bX^n_s\ds + K_\sigma(t,r)\sigma(r,\bX^{n}_r) + \int_r^t K_\sigma(t,s) \nabla \sigma^{(j)}(s,\bX^n_s) \Df_r \bX^n_s\D W_s^{(j)}.
    \end{align*}
    We know, for instance from~\cite[Theorem 3.1]{zhang2010stochastic}, that this Picard scheme converges: $\lim_{n\to\infty}\EE\abs{\bX^n_t-\bX_t^{\bx}}^2=0$. Noting that~$\psi_0=0$, Eq. \eqref{eq:psi_ineq} entails
    \begin{align*}
        \sup_{n\ge0}\psi_{n+1}(r,t) &\le c_1 \KH(t,r)^2 + c_2\int_r^t \KH(t,s)^2 \sup_{n\ge0}\psi_n(r,s)\ds \\
        &=c_1 \KH(t,r)^2 + c_2\int_r^t \KH(t,s)^2 \sup_{n\ge0}\psi_{n+1}(r,s)\ds,
    \end{align*}
    and invoking the generalised Gr\"onwall lemma we obtain
    \begin{align}\label{eq:psin_bound}
        \sup_{n\ge1}\psi_n(r,t)\le c_1 \KH(t,r)^2 + c_1 \int_r^t R_{c_2,H}(t,s)\KH(s,r)^2\dr<\infty,
    \end{align}
    where we recall that $R_{c_2,H}$ is the resolvent of $c_2K_H^2$. 
    The uniform bound of the derivatives and the $L^2$-convergence entail that~$\bX^{\bx}_t\in\DD^{1,2}$. It suffices to differentiate both sides of~\eqref{eq:main_SVE} to obtain~\eqref{eq:Malliavin_SVE}. Furthermore, $\EE\abs{\Df_r \bX_t}^2$ inherits the estimate~\eqref{eq:psin_bound}; since the integral is uniformly bounded over $0\le r<t\le T$ we obtain $\EE\abs{\Df_r \bX_t}^2\le C K_H(t,r)^2$ for some constant $C>0$. In the case $r\ge t$ it is straightforward that $\bX^n_t=0$ for all $n\in\NN$ and thus $\bX_t=0$ as well. 
\end{proof}


\bibliographystyle{alpha}
\bibliography{bib}

\end{document}